\documentclass[12pt,a4paper]{amsart}

\usepackage{amsmath, amsfonts, amsthm,amsxtra,amssymb}

\usepackage{hyperref}
\newcommand*{\mailto}[1]{\href{mailto:#1}{\nolinkurl{#1}}}

\newtheorem{theorem}{Theorem}[section]
\newtheorem{lemma}{Lemma}[section]
\newtheorem{remark}{Remark}[section]
\newtheorem{corollary}{Corollary}[section]
\newtheorem{example}{Example}[section]
\newtheorem{definition}{Definition}[section]

\def\R{\mathbb R}
\def\C{\mathbb C}
\def\N{\mathbb N}
\def\I{\mathrm i}
\def\E{\mathrm e}
\def\gt{\mathfrak{t}}
\def\gD{\mathfrak{D}}
\def\ga{\mathfrak{a}}
\def\cI{\mathcal{I}}
\def\cN{\mathcal{N}}

\newcommand{\be}{\begin{equation}}
\newcommand{\ee}{\end{equation}}
\newcommand{\nn}{\nonumber}

\def\im{\mathrm{Im}}
\def\re{\mathrm{Re}}

\def\sgn{\mathrm{sgn}}
\def\loc{\mathrm{loc}}
\def\dom{\mathrm{dom}}
\def\Span{\mathrm{span}}

\numberwithin{equation}{section}

\begin{document}

\title[The HELP inequality and indefinite spectral problems]{On a necessary aspect for the Riesz basis property for indefinite Sturm-Liouville problems}

\author[A.~Kostenko]{Aleksey Kostenko}
\address{Institute of Applied Mathematics and Mechanics\\
NAS of Ukraine\\ R. Luxemburg str. 74\\
Donetsk 83114\\ Ukraine\\ and \\
Faculty of Mathematics\\ University of Vienna\\
Nordbergstrasse 15\\ 1090 Wien\\ Austria}

\email{\mailto{duzer80@gmail.com; Oleksiy.Kostenko@univie.ac.at}}

\thanks{{\it The research was funded by the Austrian Science Fund (FWF): M1309--N13}}


\begin{abstract}
In 1996, H. Volkmer observed that the inequality
\[
\left(\int_{-1}^1\frac{1}{|r|}|f'|dx\right)^2 \le K^2 \int_{-1}^1|f|^2dx\int_{-1}^1\Big|\Big(\frac{1}{r}f'\Big)'\Big|^2dx
\]
is satisfied with some positive constant $K>0$ for a certain class of functions $f$ on $[-1,1]$ if 
the eigenfunctions of the problem
\[
-y''=\lambda\, r(x)y,\quad y(-1)=y(1)=0
\]
form a Riesz basis of the Hilbert space $L^2_{|r|}(-1,1)$. Here the weight $r\in L^1(-1,1)$ is assumed to satisfy  $xr(x)>0$ a.e. on $[-1,1]$.

We present two criteria in terms of Weyl--Titchmarsh $m$-functions for the Volkmer inequality to be valid.
Using these results we show that this inequality is valid if the operator associated with the spectral problem satisfies the linear resolvent growth condition. In particular,  we show that the Riesz basis property of eigenfunctions is equivalent to the linear resolvent growth if $r$ is odd.   
\end{abstract}

\maketitle

\section{Introduction}\label{sec:intro}
After the paper \cite{ev72} by W.N. Everitt, the integral inequality, now known as the Hardy--Littlewood--Polya--Everitt (HELP) inequality,
\be\label{eq:iHELP}
\left( \int_0^b (p|f'|+q|f|^2)dx\right)^2\le K^2 \int_0^b |f|^2wdx\int_0^b \Big|\frac{1}{w}\big((pf')'+qf\big)\Big|^2wdx,\quad (f\in\gD_{\max}),
\ee 
became one the most extensive area of research in spectral theory of Sturm--Liouville equations. Here $K$ is a positive constant; the coefficients $p^{-1},q,w\in L^1_{\loc}[0,b)$ are real valued and $w$ is assumed to be positive on $[0,b)$; $\gD_{\max}$ is the maximal linear manifold of functions for which both integrals on the right-hand side of \eqref{eq:iHELP} are finite. The famous Hardy--Littlewood inequality \cite[Chapter VII]{hlp} is a special case of \eqref{eq:iHELP} with $K=2$, $b=+\infty$, $p=w\equiv1$ and $q\equiv 0$ on $\R_+:=[0,+\infty)$. 

In \cite{ev72}, Everitt connected the above inequality with the Weyl--Titchmarsh $m$-function of the Sturm--Liouville differential equation
\be\label{eq:iSL}
-(p(x)f')'+q(x) f=\lambda\, w(x) y, \quad x\in[0,b).
\ee 
Under the assumptions $b=+\infty$, $w\equiv 1$ on $\R_+$ and $\eqref{eq:iSL}$ is regular at $x=0$ and  strong limit point at $+\infty$, Everitt obtained beautifull necessary and sufficient condition for the validity of the HELP inequality in terms of the  $m$-function associated with \eqref{eq:iSL} (see e.g. Theorem \ref{th:ev71} below). Moreover, the best possible value of $K$ and all cases of equality in \eqref{eq:iHELP} are indicated in terms of $m$. The proof in \cite{ev72} follows the line of one of the Hardy--Littlewood proofs and of course the analysis of \cite{ev72} extends to a wider setting: for the case of nonconstant  $w$ see \cite{ee82}, the case of a regular endpoint $b$ or, more general, the limit circle case at $b$ is addressed in \cite{Ben84, Ben87} and \cite{ee91}. Note also that Evans and Zettl \cite{ez78} found a general operator theoretic approach to \eqref{eq:iHELP} (see also \cite{ee82}). The latter allows to study the inequalities of the type \eqref{eq:iHELP} for other types of differential and difference operators, operators on trees etc. For further information on HELP type inequalities we refer to \cite{Ben84, Ben87, br08, ee82, ee91} (see also references therein).     

Another extensive area of research is concerned with the basis properties of (generalized) eigenfunctions of the problem
\be\label{eq:i_sp}
-y''=\lambda\, r(x)y,\quad x\in [-1,1]; \qquad y(-1)=y(1)=0.
\ee 
It is assumed that $r\in L^1(-1,1)$ and $xr(x)>0$ a.e. on $[-1,1]$, i.e., $r$ changes sign at $x=0$. It is well known that the spectrum of this problem is real and discrete, its eigenvalues are simple and accumulate at both $+\infty$ and $-\infty$. However, the eigenfunctions of \eqref{eq:i_sp} are not orthogonal in the Hilbert space $L^2_{|r|}(-1,1)$. Motivated by various problems arising in physics, scattering and transport theory, the problem of whether or not the eigenfunctions of \eqref{eq:i_sp} form a Riesz basis of $L^2_{|r|}(-1,1)$ attracted a lot of attention since the mid of seventies of the last century (see e.g. \cite{Be85, BF, BF2, CFK, KarKos, KKM09, KM08, AK11, Par03, Pya89, Vol_96}). The first general sufficient condition for the Riesz basis property was obtained by Beals in \cite{Be85} and later it has been extended and generalized by many authors (for a survey we refer to the recent papers \cite{BF2, KKM09, AK11}).  

For a long time the following question remained open: {\em are there weights $r\in L^1(-1,1)$ such that the eigenfunctions of \eqref{eq:i_sp} do not form a Riesz basis of $L^2_{|r|}(-1,1)$?}  It was answered in the affirmative by H. Volkmer in \cite{Vol_96}. Namely,  Volkmer found the following connection between the Riesz basis property for \eqref{eq:i_sp} and HELP type inequalities.
\begin{theorem}[Volkmer]\label{lem:vol}
Let $r\in L^1(-1,1)$
\footnote{In \cite{Vol_96}, Theorem \ref{lem:vol} was established under the assumption $r\in L^\infty(-1,1)$. It is noticed in \cite[Theorem 3.1]{BF2} that the statement remains true for $L^1$ weights.} 
and $xr(x)>0$ a.e. on $[-1,1]$. 
If the eigenfunctions of the problem \eqref{eq:i_sp}
form a Riesz basis of $L^2_{|r|}(-1,1)$, then there is $K>0$ such that
\be\label{eq:vol}
\left(\int_{-1}^1\frac{1}{|r|}|f'|dx\right)^2 \le K^2 \int_{-1}^1|f|^2dx\int_{-1}^1\Big|\big(\frac{1}{r}f'\big)'\Big|^2dx,\quad (f\in\dom(A)),
\ee
where 
\[
\dom(A)=\{f\in L^2(-1,1):\ f, r^{-1}f'\in AC[-1,1],\ (r^{-1}f')(\pm 1)=0,\ (r^{-1}f')'\in L^2(-1,1)\}.
\]
If in addition $r$ is odd, $r(x)=-r(-x)$ a.e. on $[-1,1]$, then the inequality 
\be\label{eq:iHELP+}
\left( \int_0^1 \frac{1}{r}|f'|dx\right)^2\le K^2 \int_0^1 |f|^2dx\int_0^1\Big|\big(\frac{1}{r}f'\big)'\Big|^2rdx,\quad (f\in\dom(A_+)),
\ee
is valid, i.e., there is $K>0$ such that \eqref{eq:vol} holds for all $f\in\dom(A_+)$, if the eigenfunctions of \eqref{eq:i_sp} form a Riesz basis of $L^2_{|r|}(-1,1)$. Here 
\[
\dom(A_+):=\{f\in L^2(0,1): \, f, r^{-1}f'\in AC[0,1], \ (r^{-1}f')(1)=0,\ (r^{-1}f')'\in L^2(0,1)\}.
\]
\end{theorem}
Noting that there are weights such that \eqref{eq:iHELP+} is not valid, Volkmer gave a positive answer to the existence problem. Moreover, using a Baire category argument, it is noticed in \cite{Vol_96} that, in general,  the eigenfunctions of \eqref{eq:i_sp} do not form a Riesz basis of $L^2_{|r|}$ if $r$ is odd. Concrete examples of odd weights were given later by Fleige, Abasheeva and Pyatkov (we refer for details to \cite{BF2}). Moreover, using Pyatkov's interpolation criterion \cite{Pya89}, Parfenov \cite{Par03} found a necessary and sufficient condition for the Riesz basis property under the assumption that $r$ is odd. Notice that the problem on the Riesz basis property for \eqref{eq:i_sp} is still open if the oddness assumption is dropped. The most recent results can be found in \cite{BF2, CFK} (see also references therein).

The main objective of this paper is to investigate the inequality of Volkmer \eqref{eq:vol} in the general case, i.e., without the assumption that $r$ is odd. Our main aim is to find a criterion for the validity of \eqref{eq:vol} in terms of Weyl--Titchmarsh $m$-functions. We are motivated by the papers \cite{Ben84, Ben87}, where Bennewitz gave a necessary and sufficient condition for the validity of \eqref{eq:iHELP+} in terms of the weight $r$ (see Theorem \ref{Bennewitz_HELP}). His proof is based on the analysis of the asymptotic behavior of the corresponding Weyl--Titchmarsh $m$-function (see Section \ref{sec:help} for details). It is interesting to note that the class of weights such that \eqref{eq:iHELP+} is valid coincides with the class of odd weights such that \eqref{eq:i_sp} has a Riesz basis property, i.e., {\em if $r$ is odd, then the eigenfunctions of \eqref{eq:i_sp} form a Riesz basis of $L^2_{|r|}(-1,1)$ if and only if \eqref{eq:iHELP+} is valid} (however, the latter equivalence was first observed in \cite{BF}\footnote{It seems that the paper \cite{Ben87} is not widely known among authors studying the Riesz basis property of eigenfunctions for indefinite spectral problems. In particular, in \cite{BF}, this equivalence was established by using a different argument.}, see also \cite{BF2}). 

We present two criteria for the validity of \eqref{eq:vol} in terms of $m$-functions associated with \eqref{eq:i_sp}. The first criterion (Theorem \ref{th:crit_mf}) is formulated in the form similar to the classical Everitt criterion. This criterion also provides the best possible constant $K$ in \eqref{eq:vol} in terms of $m$-coefficients. However, Everitt type conditions require the knowledge of  the asymptotic behavior of $m$-functions in some sector in the upper half-plane $\C_+$, which contains the imaginary semi-axis $\I\R_+$.  It turns out that to give the answer on whether or not \eqref{eq:vol} is valid it suffices to know the behavior of $m$-functions along $\I \R_+$. This  is the content of our main result, Theorem \ref{th:criterion}. Namely, let $m_+$ and $m_-$ be the $m$-functions corresponding to \eqref{eq:i_sp} on $(0,1)$ and $(-1,0)$, respectively (for definitions see Section \ref{ss:2.01}). {\em Then \eqref{eq:vol} is valid if and only if }
\be\label{eq:i_crit}
\sup_{y>0}\frac{\re (m_+(\I y)+m_-(\I y))}{|m_+(\I y)-m_-(-\I y)|}<+\infty.
\ee
The latter in particular implies that \eqref{eq:iHELP+} is valid precisely if 
\be\label{eq:i_crit+}
\sup_{y>0}\frac{\re\, m_+(\I y)}{\im \, m_+(\I y)}<+\infty.
\ee
To the best of our knowledge this criterion for the validity of the HELP inequality \eqref{eq:iHELP+} seems to be new.
  
Let us also note that in the series of papers \cite{KarKos, KKM09, KM08} several necessary and sufficient conditions, formulated  in terms of $m$-functions, for the Riesz basis property of eigenfunctions of \eqref{eq:i_sp} have been found. In particular, it is shown in \cite{KarKos} that the condition
\be\label{eq:i_lrg}
\sup_{\lambda\in\C_+}\frac{\im (m_\pm^r(\lambda))}{|m_+^r(\lambda)+ m_-^r(-\lambda)|}<\infty
\ee
is necessary for the Riesz basis property of eigenfunctions of \eqref{eq:i_sp}. Here $m_+^r$ and $m_-^r$ are the $m$-functions associated with the problem \eqref{eq:i_sp} (see Section \ref{sec:LRG} for further details). 
On the other hand (see \cite{KarKos}),  \eqref{eq:i_lrg} holds true if the operator $H$ (see \eqref{eq:H}) associated with \eqref{eq:i_sp} satisfies {\em the linear resolvent growth  (LRG)  condition:}  
\[
\|(H - \lambda)^{-1}\|\le \frac{C}{|\im \, \lambda|},\quad  (\lambda\in\C\setminus\R).
\]
Here $C>0$ is a positive constant independent of $\lambda$. Let us also mention that it is shown by M. Malamud and the author \cite{Mal} that \eqref{eq:i_lrg} is also sufficient for the linear growth of the resolvent of $H$ and, moreover, the analysis extends to an abstract operator theoretic  setting. 

Noting that the functions $m_\pm$ and $m_\pm^r$ are connected by (see \eqref{eq:m_conn})
\[
m_\pm^r(\lambda)=\lambda\, m_\pm(\lambda),\quad (\lambda\in\C_+),
\]
we see that \eqref{eq:i_crit} follows from \eqref{eq:i_lrg}. 
Therefore, {\em Volkmer's inequality \eqref{eq:vol} is valid if the operator $H$ associated with \eqref{eq:i_sp} satisfies the linear resolvent growth condition}. However, in the case of odd weights $r$, \eqref{eq:vol} becomes also sufficient for the Riesz basis property and hence we conclude that {\em the linear resolvent growth  implies the Riesz basis property for \eqref{eq:i_sp} if $r$ is odd} (see Theorem \ref{th:sim_odd}).

In conclusion, let us briefly outline the content of the paper.  Section \ref{sec:pre} is of preliminary character. It contains necessary notions and facts on differential expression and Weyl--Titchmarsh $m$-functions. In Section \ref{sec:help}, we give an overview of results on the HELP inequality \eqref{eq:iHELP+}. We present Everitt's and Bennewitz's criteria (Theorems \ref{th:ev71} and \ref{Bennewitz_HELP}, respectively) and also prove Theorem \ref{th:ak_crit+}, which states that \eqref{eq:i_crit+} is necessary and sufficient for the validity of \eqref{eq:iHELP+}. Sections \ref{sec:vol_I} and \ref{sec:vol_II} present two criteria for the validity of the Volkmer inequality \eqref{eq:vol}. In Section \ref{sec:cond} we apply the results from previous subsections and obtain some necessary and sufficient conditions in terms of weights. It is interesting to note that for scaled odd weights \eqref{eq:tilde_r} the inequality \eqref{eq:vol} is always valid and the constant $K$ in \eqref{eq:vol} is uniform and depends only on a scaling parameter $a$ (see Lemma \ref{lem:scaling}). In the final Section \ref{sec:LRG}, we  establish a connection between the inequality \eqref{eq:vol} and the linear resolvent growth condition (Theorem \ref{th:connection}). The latter allows us to extend the list of various criteria for the Riesz basis property of \eqref{eq:i_sp} in the case of odd weights (Theorem \ref{th:sim_odd}). 
In Appendix, we present some necessary facts from \cite{Ben89} on asymptotics of Weyl--Titchmarsh $m$-functions.

{\bf Notation.} $L^1(a,b)$ and $AC[a,b]$ are the sets of Lebesgue integrable and absolutely continuous functions on a compact interval $[a,b]$; if $r\in L^1(a,b)$ is positive, then $L^2_{r}(a,b)$ stands for the Hilbert space of equivalence classes with the norm $\|f\|=\big(\int_{(a,b)}|f|^2r(x)dx\big)^{1/2}$; $L^2(a,b):=L^2_{r}(a,b)$ if $r\equiv 1$.

$\N,\R,\C$ have the standard meanings; $\C_+$ is the open upper half-plane, $\C_+=\{\lambda\in \C:\, \im\, \lambda>0\}$; $\bar{\lambda}$ is the complex conjugate of $\lambda\in\C$; $\R_+:=[0,+\infty)$ and $\I \R_+=\{\I y:\ y\in\R_+\}$.

Prime $'$ denotes the derivative, $'\equiv \frac{d}{dx}$.

The notation '$(x\in X)$' is to be read as 'for all $x$ from the set $X$'.


\section{Preliminaries}\label{sec:pre}
\subsection{Differential operators}\label{ss:2.01}
Consider the following differential expressions
\be\label{eq:de}
\mathfrak{a}[f]:=-\big(\frac{1}{r}f'\big)',\qquad \ell[f]:=-\big(\frac{1}{|r|}f'\big)'.
\ee
In $L^2(-1,1)$, one associates with these expressions the following operators
\be\label{eq:oper}
Af=\ga[f],\quad f\in \dom(A);\qquad Lf=\ell[f],\quad f\in \dom(L),
\ee
where
\be\label{eq:dom_a}
\dom(A)=\{f\in L^2(-1,1):\ f,r^{-1}f'\in AC[-1,1],\ (r^{-1}f')(\pm 1)=0,\ \ga[f]\in L^2\},
\ee
\be\label{eq:dom_l}
\dom(L)=\{f\in L^2(-1,1):\ f,|r|^{-1}f'\in AC[-1,1], \ (|r|^{-1}f')(\pm 1)=0,\ \ell[f]\in L^2\}.
\ee
Consider also the minimal and maximal domains
\be\label{eq:dom_min}
\mathfrak{D}_{\min}=\{f\in \dom(L):\ f(0)=(|r|^{-1}f')(0)=0\},
\ee
and
\be\label{eq:dom_max}
\mathfrak{D}_{\max}=\{f\in L^2(-1,1):\ f, |r|^{-1}f'\in AC(\cI\setminus\{0\}),\ (|r|^{-1}f')(\pm 1)=0,\ \ell[f]\in L^2\}.
\ee
Define the operators
\be\label{eq:LA}
L_{\min}f=\ell[f], \quad A_{\min}f=a[f],\quad \dom(L_{\min})=\dom(A_{\min})=\mathfrak{D}_{\min},
\ee
and
\[
A_{\max}f=a[f], \quad A_{\max}f=a[f],\quad \dom(L_{\max})=\dom(A_{\max})=\mathfrak{D}_{\max}.
\]
Note that the operators $L_{\min}$ and $A_{\min}$ are symmetric, $n_\pm(L_{\min})=n_\pm(A_{\min})=2$, and
\[
L_{\min}^*=L_{\max},\qquad A_{\min}^*=A_{\max}.
\]
Moreover, 
\[
A_{\min}=JL_{\min},\quad A_{\max}=JL_{\max},
\]
where $J:f(x)\to (\sgn\, x) f(x)$.

\subsection{Weyl--Titchmarsh $m$-functions}\label{ss:WF}
Let $c(x,\lambda)$ and $s(x,\lambda)$ be the solutions of $\ell[y]=\lambda y$ satisfying the initial conditions
\be\label{eq:fs}
c(0,\lambda)=(|r|^{-1}s')(0,\lambda)=1,\quad s(0,\lambda)=(|r|^{-1}c')(0,\lambda)=0.
\ee
Define the Weyl solutions corresponding to the Neumann boundary conditions at $x=\pm 1$
\be\label{eq:weyl_s}
 \psi_\pm(x,\lambda)=s(x,\lambda)\mp m_\pm(\lambda)c(x,\lambda),\quad (|r|^{-1}\psi_\pm')(\pm 1,\lambda)=0,\qquad (\lambda\in\C_+).
\ee
The functions $m_+$ and $m_-$ are called the $m$-functions corresponding to $\ell$ on $[0,1]$ and $[-1,0]$, respectively. Notice that
\be\label{eq:psinorm}
\|\psi_\pm(x,\lambda)\chi_\pm(x)\|^2=\frac{\im\, m_\pm(\lambda)}{\im\, \lambda},\quad \lambda\in \C_+,
\ee
where $\chi_\pm$ is the indicator function of $\cI_\pm$, $\cI_+=(0,1)$ and $\cI_-=(-1,0)$. 
The latter means that $m_\pm$ is a Herglotz function. Moreover, the function $m_\pm$ admits the representation
\be\label{eq:m_repr}
m_\pm(\lambda)=\int_{\R_+}\frac{d\tau_\pm(s)}{s-\lambda},\quad \lambda\notin\R_+,
\ee
where the positive measures $d\tau_+$ and $d\tau_-$, called the spectral measures,  satisfy
\be\label{eq:m_repr2}
\int_{\R_+}\frac{d\tau_\pm(s)}{1+s}<\infty,\quad \int_{\R_+}d\tau_\pm(s)=\infty.
\ee
In particular, \eqref{eq:m_repr} means that $m_+$ and $m_-$ belong to the Krein--Stieltjes class $S$ (see \cite{KK1}).

Notice also that the functions $m_+$ and $m_-$ are meromorphic and both have a simple pole at $\lambda=0$. Indeed, the singularities of $m_+$ and $m_-$ are precisely the spectra of the problems
\begin{eqnarray*}
&-(|r|^{-1}y')'=\lambda \, y,\quad x\in \cI_\pm,\\
& (|r|^{-1}y')(0)=(|r|^{-1}y')(\pm 1)=0,
\end{eqnarray*}
and $\lambda =0$ is the eigenvalue with the eigenfunction $y_\pm(x,0)=c(x,0)\equiv 1$.

Further, note that the deficiency subspaces of $L_{\min}$ and $A_{\min}$ are given by
\be\label{eq:defect_L}
           \cN_\lambda(L_{\min})=\Span\{\psi_+(x,\lambda)\chi_+(x),\psi_-(x,\lambda)\chi_-(x)\}
\ee
 and
\be\label{eq:defect_A}
\cN_\lambda(A_{\min})=\Span\{\psi_+(x,\lambda)\chi_+(x),\psi_-(x,-\lambda)\chi_-(x)\},\quad \lambda\in\C\setminus\R.
\ee
Finally, by the von Neumann formula, the maximal domain admits the representation
\be\label{eq:neum}
\gD_{\max}=\gD_{\min}\dotplus\cN_{\lambda}\dotplus\cN_{\overline{\lambda}},\quad \lambda\in\C_+.
\ee


\section{The HELP inequality: the regular case}\label{sec:help}

Assume that $r\in L^1(0,1)$ is positive a.e. on $(0,1)$. 
Consider the following inequality
\be\label{eq:help}
\Big(\int_0^1\frac{1}{r}|f'|^2dx\Big)^2\le K^2\int_0^1|f|^2dx\, \int_0^1\big|\big(\frac{1}{r}f'\big)'\big|^2dx,\qquad (f\in\dom(A_+)),
\ee
where 
\be\label{eq:dom+}
\dom(A_+)=\{f\in L^2(0,1):\ f,r^{-1}f'\in AC[0,1],\ (r^{-1}f')(1)=0,\ (r^{-1}f')'\in L^2\}.
\ee
The inequality \eqref{eq:help} is said to be valid if there is $K>0$ such that \eqref{eq:help} holds for all $f\in\dom(A_+)$.

Let $m_+$ be the $m$-function defined by \eqref{eq:weyl_s}. The following criterion for the validity of \eqref{eq:help} was found by Everitt \cite{ev72} (see also \cite{ee82}, where the regular case was treated).
\begin{theorem}[Everitt]\label{th:ev71}
The inequality \eqref{eq:help} is valid if and only if there is $\theta\in(0,\frac{\pi}{2})$ such that
\be\label{eq:ev71}
-\im (\lambda^2\, m_+(\lambda))\ge 0,\quad (\lambda \in \Gamma_{\theta}),
\ee
where $\Gamma_{\theta}:=\{z\in\C_+: \frac{\re z}{|z|}\in [-\cos\theta,\cos\theta]\}$.

Moreover, the best possible $K$ in \eqref{eq:help} is given by 
\[
K=\frac{1}{\cos \theta_0},\quad \theta_0:=\inf \big\{\theta\in \big(0,\frac{\pi}{2}\big]:\ \eqref{eq:ev71}\, \text{is satisfied}\big\}.
\] 
\end{theorem}
Bennewitz in \cite{Ben87} found necessary and sufficient condition on the coefficient $r$ such that \eqref{eq:help} is valid. 
\begin{definition}
Let $r\in L^1(0,1)$ be positive a.e. on $(0,1)$. We say that $r$ satisfies Bennewit's condition at $0$  if there is $t\in (0,1)$ such that 
\be\label{Parfenov}
S_0(t):=\limsup_{x\to 0}S(t,x)\neq 1,\quad \text{where} \quad 
S(t,x):=\frac{\int_{[0, tx]}r(\zeta)d\zeta}{\int_{[0, x]}r(\zeta)d\zeta}.
\ee 
\end{definition}
\begin{theorem}[Bennewitz]\label{Bennewitz_HELP}
The inequality \eqref{eq:help} is valid if and only if $r$ satisfies Bennewitz's condition.
\end{theorem}
The proof of Theorem \ref{Bennewitz_HELP} is based on the following result for $m$-functions.
\begin{lemma}[\cite{Ben87}]\label{lem:Bennewitz}
Let $m_+$ be the m-function defined by \eqref{eq:weyl_s}.
Then
\be\label{eq:m_ass}
|m_+(\lambda)|=O\big(\im\, m_+(\lambda)\big) \quad \text{as}\quad \lambda \to \infty
\ee
 in any nonreal sector (a sector non intersecting the real axis) if $r$ satisfies Bennewitz's condition \eqref{Parfenov}.
%
\end{lemma}
As we shall show below, the converse statement is also true (see Corollary \ref{cor:ben}).

Everitt's criterion for the validity of \eqref{eq:help} requires the knowledge of asymptotic behavior of the corresponding $m$-function $m_+$ at least in some sector of $\C_+$, which contains the imaginary semi-axis $\I\R_+$. Our main aim is to show that it suffices to know only the behavior of $m_+$ along the ray $\I\R_+$.
 \begin{theorem}\label{th:ak_crit+}
 Let $m_+$ be the $m$-function defined by \eqref{eq:weyl_s}. Then the inequality \eqref{eq:help} is valid if and only if
 \be\label{eq:crit+}
 \sup_{y>0}\frac{\re\, m_+(\I y)}{\im\, m_+(\I y)}<\infty.
 \ee
 \end{theorem}
 
 Before proving Theorem \ref{th:ak_crit+} we need the following result.
 \begin{lemma}\label{lem:m_assympt}
Assume that $S_0\equiv 1$ on $(0,1)$. Then there is a sequence $\{\lambda_j\}\subset\C_+$ such that 
\be\label{eq:lam_01}
\lambda_j=(k_j+\I)y_j, \quad y_j\to+\infty,\quad k_j\to +0, 
\ee
and $\arg m_+(\lambda_j)=o(1)$ as $ j\to+\infty$, i.e.
\be\label{eq:lam_02}
\frac{\im\, m_+(\lambda_j)}{\re\, m_+(\lambda_j)}=o(1),\quad j\to\infty.
\ee
\end{lemma}
\begin{proof}
Let $\lambda=\rho\E^{\I\theta}\in\C_+$. Denote also $m_+(\lambda)=|m_+(\lambda)|\E^{\I \theta_m}$. Note that $\theta_m\in(0,\pi)$ if $\lambda\in\C_+$ since $m_+$ is Herglotz. Then
\begin{align*}
\im (\lambda^2 m_+(\lambda))
=\rho^2|m_+|\sin( 2\theta+\theta_m).
\end{align*}
Therefore, $\im (\lambda^2 m_+(\lambda))>0$ precisely when $2\theta+\theta_m<\pi$.

If $S_0\equiv 1$ on $(0,1)$, then, by Theorem \ref{th:ev71} and Theorem \ref{Bennewitz_HELP}, there are sequences $\{\theta_j\}_1^\infty \subset (0,\frac{\pi}{2})$ and $\{\rho_j\}_1^\infty\subset \R_+$ such that $\rho_j\uparrow \frac{\pi}{2}$ and $\im (\lambda_j^2m_+(\lambda_j))>0$, where $\lambda_j:=\rho_j\E^{\I \theta_j}$, $(j\in\N)$. The latter means that $2\theta_j+(\theta_m)_j<\pi$, where $(\theta_m)_j:=\arg m_+(\rho_j\E^{\I\theta_j})\in (0,\pi)$. Therefore, $(\theta_m)_j\downarrow 0$ as $j\to \infty$. 

To complete the proof it remains to note that $\lambda_j$ can accumulate only at $0$ or at $\infty$ since $m_+$ is Herglotz. However, $m_+$ has a pole at $\lambda=0$ 
and hence $\lambda_j$ goes to $\infty$.
\end{proof}
As an immediate corollary we obtain the statement converse to Lemma \ref{lem:Bennewitz}. 
\begin{corollary}\label{cor:ben01}
Let $m_+$ be the m-function defined by \eqref{eq:weyl_s}. If $m_+$ satisfies \eqref{eq:m_ass} in any nonreal sector, then $S_0\not\equiv 1$.
\end{corollary}
Combining Theorem \ref{Bennewitz_HELP} with Lemma \ref{lem:Bennewitz} and Corollary \ref{cor:ben01}, we arrive at another criterion for the validity of \eqref{eq:help} in terms of the Weyl--Titchmarsh coefficient $m_+$. 
\begin{corollary}\label{cor:ben}
The inequality \eqref{eq:help} is valid if and only if $m_+$ satisfies \eqref{eq:m_ass} in any nonreal sector in $\C_+$.
\end{corollary}
\begin{remark}
By Corollary \ref{cor:ben}, the inequality \eqref{eq:help} is valid if and only if $m_+$ maps any nonreal sector into a nonreal sector. However, Theorem \ref{th:ak_crit+} states that for the validity of \eqref{eq:help} it suffices to check that the image of $\I \R_+$ under $m_+$  lies in some nonreal sector. In particular, the following are equivalent:
\begin{enumerate}
\item[(i)] $m_+$ maps any nonreal sector into a nonreal sector,
\item[(ii)] $m_+$ maps $(\I 0, +\I\infty)$ into a nonreal sector.
\end{enumerate}
 \end{remark}

 
 \begin{proof}[Proof of Theorem \ref{th:ak_crit+}]
 {\em Necessity.} Assume that  \eqref{eq:help}is valid. Firstly, note that the Weyl solution $\psi_+(x,\lambda)$ defined by \eqref{eq:weyl_s} belongs to $\dom(A_+)$. Using \eqref{eq:psinorm} and \eqref{eq:weyl_s} we get
 \begin{eqnarray*}
 & \int_0^1 |\psi_+(x,\I y)|^2dx=\frac{1}{y}\im\, m_+(\I y),\\
 & \int_0^1 \big|\big(\frac{1}{r}\psi_+'(x,\I y)\big)'\big|^2dx=y^2\int_0^1 |\psi_+(x,\I y)|^2dx=y\im\, m_+(\I y),\\
 &\int_0^1 \frac{1}{r}|\psi_+'(x,\I y)|^2dx=\psi_+(x,\I y)\big(r^{-1}\psi_+'(x,-\I y)\big)|_{x=0}^1-\I y\int_0^1|\psi_+(x,\I y)|^2dx\\
 &=m_+(\I y)-\I \im\, m_+(\I y)=\re\, m_+(\I y).
 \end{eqnarray*}
 Therefore, substituting $\psi_+(x,\I y)$ into \eqref{eq:help}, we arrive at
 \[
 \re\, m_+(\I y) \le K\im\,  m_+(\I y),\quad (y>0).
 \]
 
 {\em Sufficiency.} Assume the converse, i.e., \eqref{eq:help} is not valid. Then, by Theorem \ref{Bennewitz_HELP}, $S_0\equiv 1$ on $(0,1)$ and hence, by Lemma \ref{lem:m_assympt}, there is a sequence $\{\lambda_j\}\subset \C_+$ with the properties \eqref{eq:lam_01}--\eqref{eq:lam_02}.
  
  Using \eqref{eq:m_repr}, observe that for $\lambda_j=x_j+\I y_j=(k_j+\I)y_j$
\[
\im\, m_+(\lambda_j)-\im\, m_+(\I y_j)=\int_{\R_+}\frac{2sx_j - x_j^2}{s^2+y_j^2}\frac{y_j}{(s-x_j)^2+y_j^2}d\tau_+(s).
\]
Since
\[
\frac{|2sx_j-x_j^2|}{s^2+y_j^2}\le \frac{2sx_j+x_j^2}{s^2+y_j^2}\le \frac{x_j}{y_j}+\frac{x_j^2}{y_j^2}=k_j+k_j^2\le 2k_j,\quad k_j\le1,
\]
we get
\be\label{eq:5.07}
\big|\im\, m_+(\lambda_j)-\im\, m_+(\I y_j)\big|\le 2k_j\im\, m_+(\lambda_j)\quad k_j\le1.
\ee
Further,  
\[ 
\re\, m_+(\lambda_j)-\re\, m_+(\I y_j)=
\int_{\R_+}\Big(\frac{s-x_j}{(s-x_j)^2+y_j^2}-\frac{s}{s^2+y_j^2}\Big)d\tau_+(s).
\]
Note that
\[
\Big|\frac{s-x_j}{(s-x_j)^2+y_j^2}-\frac{s}{s^2+y_j^2}\Big|\le \frac{s^2 x_j +sx_j^2 +x_jy_j^2}{(s^2+y_j^2)((s-x_j)^2+y_j^2)}\le (2k_j+k_j^2)\frac{y_j}{(s-x_j)^2+y_j^2}.
\]
Thus,  we get
\be\label{eq:5.08}
\big|\re\, m_+(\lambda_j)-\re\, m_+(\I y_j)\big|\le 
3k_j \im\, m_+(\lambda_j)\quad k_j\le1.
\ee
Therefore, combining \eqref{eq:5.07}, \eqref{eq:5.08} with \eqref{eq:lam_02} and noting that $k_j\downarrow 0$, we obtain
\[
\im\, m_+(\I y_j)=o\big(\re\, m_+(\I y_j)\big),\quad j\to\infty.
\]
Therefore, \eqref{eq:crit+} is not satisfied. 
The proof is completed.
 \end{proof}
 \begin{remark}
 According to the proof of necessity of \eqref{eq:crit+} for the validity of \eqref{eq:help}, Theorem \ref{th:ak_crit+} means that it suffices to check \eqref{eq:help} on the Weyl solutions corresponding to imaginary $\lambda=\I y$, $(y>0)$.  That is, {\em \eqref{eq:help} is valid if and only if there is $K>0$ such that \eqref{eq:help} holds true for all $f=\psi_+(x,\I y)$, $y>0$}.
 \end{remark}
 Note that Theorem \ref{th:ak_crit+}, as well as Corollary \ref{cor:ben}, does not provide the best possible value of $K$ in \eqref{eq:help}. However, it gives a lower bound for $K$.
 \begin{corollary}
 Let $K_0$ be the best possible value of $K$ in \eqref{eq:help}. Then
 \[
  \sup_{y>0}\frac{\re\, m_+(\I y)}{\im\, m_+(\I y)}\le K_0.
 \]
 \end{corollary}
\begin{proof}
The claim immediately follows from the proof of necessity of Theorem \ref{th:ak_crit+}.
%
\end{proof}
\begin{remark}
Theorem \ref{th:ak_crit+} is concerned with the particular case of the general HELP inequality \eqref{eq:iHELP}. However, Theorem \ref{th:ak_crit+} remains true under the following assumptions on coefficients in \eqref{eq:iHELP}: 
\begin{enumerate}
\item[(i)] \eqref{eq:iSL} is regular at both endpoints; 
\item[(ii)] the spectral problem \eqref{eq:iSL} subject to the Neumann boundary conditions has a nonnegative spectrum; 
\item[(iii)] the functions from $\gD_{\max}$ also satisfy the Neumann boundary condition at $x=b$.  
\end{enumerate}

The case of a singular endpoint $x=b$ will be considered elsewhere. 
\end{remark}
 

\section{Volkmer's inequality: the first criterion}\label{sec:vol_I}


Assume now that $r\in L^1(-1,1)$ is real-valued and $xr(x)>0$ a.e. on $(-1,1)$. 
Consider the following inequality
\be\label{eq:volkmer}
\Big(\int_{-1}^1\frac{1}{|r|}|f'|^2dx\Big)^2\le K^2 \int_{-1}^1|f|^2dx\, \int_{-1}^1\big|\big(\frac{1}{r}f'\big)'\big|^2dx,\qquad (f\in\dom(A)).
\ee
Here $\dom(A)$ denotes the domain of the operator $A$ and is given by \eqref{eq:dom_a},
\[
\dom(A)=\{f\in L^2(-1,1):\, f,\, r^{-1}f'\in AC[-1,1],\, (r^{-1}f')(\pm 1)=0,\ \ga[f]\in L^2\}.
\]
The inequality \eqref{eq:volkmer} is said to be valid if there is $K>0$ such that \eqref{eq:volkmer} holds for all $f\in\dom(A)$.

Note that the inequality \eqref{eq:volkmer} differs from \eqref{eq:iHELP} 
since the left-hand side in \eqref{eq:volkmer} is the Dirichlet integral corresponding to the operator $L$ (see \eqref{eq:LA}), however, we consider \eqref{eq:volkmer} on functions from $\dom(A)$.   Clearly, the inequality \eqref{eq:volkmer} considered on functions $f\in \dom(L)$ holds true with $K=1$ (this follows from integration by parts applied to the left-hand side of \eqref{eq:volkmer} and subsequent use of the Cauchy--Schwarz inequality).

On the other hand, if we consider \eqref{eq:volkmer} on a larger domain $\gD_{\max}$ (see \eqref{eq:dom_max}), then clearly \eqref{eq:volkmer} is equivalent to two separated HELP inequalities of the form \eqref{eq:help} and then the answer is given by criteria from Section \ref{sec:help}. However, we consider \eqref{eq:volkmer} on a domain $\dom(A)$, which is smaller than $\gD_{\max}$. Namely, $f\in\dom(A)$ precisely when $f\in\gD_{\max}$ and satisfies additional boundary conditions at $x=0$:
\be\label{eq:bc}
f(+0)=f(-0),\quad \big(r^{-1}f\big)'(+0)=\big(r^{-1}f\big)'(-0).
\ee
Therefore, conditions from Section \ref{sec:help} become only sufficient for the validity of \eqref{eq:volkmer}.

We shall present two criteria for the validity of \eqref{eq:volkmer}. Note that the first one, Theorem \ref{th:crit_mf}, is the analog of the Everitt criterion \eqref{eq:ev71} and the second criterion, Theorem \ref{th:criterion}, is the analog of Theorem \ref{th:ak_crit+}. 

Before formulate the first result, we need some notation. Let $m_+$ and $m_-$ be the $m$-functions defined by \eqref{eq:weyl_s}. 
Set 
\be\label{eq:ti_m}
\tilde{M}_\pm(\lambda):=\frac{1}{\im m_\pm(\lambda)}
\left(\begin{array}{cc}
\im\lambda &  -\im(\lambda m_\pm(\lambda))\\
-\im(\lambda m_\pm(\lambda)) & \im\lambda |m_\pm(\lambda)|^2
\end{array}\right),
\ee
and 
\be
\label{eq:ma}
M_A(\lambda):
=\tilde{M}_+(\lambda)+\tilde{M}_-(\lambda),\quad \lambda\in\C_+.
\ee
\begin{theorem}\label{th:crit_mf}
Let $m_+$ and $m_-$ be the $m$-functions defined by \eqref{eq:weyl_s}. Let also $M_A$ be given by \eqref{eq:ti_m}, \eqref{eq:ma}. Then \eqref{eq:volkmer} is valid if and only if there is $\theta \in (0,\frac{\pi}{2})$ such that
\be\label{eq:m_cond}
M_A(\lambda)\ge 0\quad \text{for all}\quad \lambda\in\C_+\quad \text{with}\quad \frac{|\re\lambda|}{|\lambda|}=\cos \theta.
\ee

Moreover, the best possible $K$ is given by 
\[
K=\frac{1}{\cos\theta_0},\qquad \theta_0:=\inf\big\{\theta\in\big(0,\frac{\pi}{2}\big]:\, \eqref{eq:m_cond}\, \text{is satisfied}\big\}.
\]
\end{theorem}

The proof is based on ideas of the operator--theoretic proof of the HELP inequality \cite[\S 8]{ee82}. Note that this method was first proposed in \cite{ez78}.  

We divide the proof in several steps.

Firstly, for $f,g\in \gD_{\max}$ consider the following bilinear form 
\[
\gt[f,g]:=\int_{-1}^1\frac{1}{|r|}f'\bar{g}'dx,\qquad \gt[f]:=\gt[f,f].
\]
Then we can rewrite \eqref{eq:volkmer} as follows
\[
\gt[f]\le K \|f\|_{L^2}\|\ga[f]\|_{L^2},\qquad (f\in \dom(A)).
\]
Clearly, for all $f,g\in \gD_{\max}$ we have $(\ell[f],\ell[g])=(\ga[f],\ga[g])$ and hence \eqref{eq:volkmer} becomes
\be\label{eq:vol2}
\gt[f]\le K \|f\|_{L^2}\|\ell[f]\|_{L^2},\qquad (f\in \dom(A)).
\ee
Further, for $\lambda\in\C_+$ let us consider the hermitian form
\be\label{eq:J}
J_\lambda(f,g)=|\lambda|^2(f,g)_{L^2}-2\re\, \lambda\, \gt[f,g]+(\ell[f],\ell[g])_{L^2},\quad f,g\in\gD_{\max}.
\ee
Noting that
\be
J_\lambda(f):=J_\lambda(f,f)=\|f\|^2\Big(|\lambda|-\frac{\re\, \lambda}{|\lambda|}\frac{\gt[f]}{\|f\|^2}\Big)^2
+\frac{1}{\|f\|^2}\Big(\|f\|^2\|\ell[f]\|^2-\frac{(\re\, \lambda)^2}{|\lambda|^2}\gt[f]^2\Big),\label{eq:j_lam}
\ee 
we immediately arrive at the following statement.
\begin{lemma}\label{lem:3.1}
The inequality  \eqref{eq:volkmer} holds true for all $f\in \gD_{\max}$ ($f\in \dom(A)$) with some $K>0$ precisely if $J_\lambda(f)$ is positive for all $f\in \gD_{\max}$ ($f\in \dom(A)$) on two rays in the upper half-plane for which $\frac{\re \lambda}{|\lambda|}=\pm \frac{1}{K}$.
\end{lemma}
The next result shows that it suffices to consider the form $J_\lambda$ on a finite dimensional subspace.
\begin{lemma}\label{lem:j_lam}
Let $\cN_\lambda$, $\lambda\in\C\setminus\R$, be the deficiency subspace \eqref{eq:defect_L}. Then the form $J_\lambda$ is nonnegative on $\gD_{\max}$ if and only if it is so on $\cN_{\lambda}\dotplus\cN_{\overline{\lambda}}$ 
\end{lemma}
\begin{proof}
By the first von Neumann formula \eqref{eq:neum}, $f\in\gD_{\max}$ admits the representation
\[
f=f_0+f_{\lambda}+f_{\overline{\lambda}},\qquad f_0\in\gD_{\min},\quad f_\lambda\in\cN_{\lambda}(L). 
\]
Further, if $f,g\in\gD_{\max}$, then 
\be\label{eq:green}
 \gt[f,g]=-(|r|^{-1}f')\bar{g}|^{+0}_{-0}+(\ell[f],g)_{L^2}= -f(|r|^{-1}\bar{g}')|^{+0}_{-0}+(f,\ell[g])_{L^2}.
\ee
Hence, by \eqref{eq:green}, $\gt[f_0]=(\ell[f_0],f_0)$ and then
\[
J_\lambda(f_0)=|\lambda|^2\|f_0\|^2-2\re \lambda (\ell[f_0],f_0)+\|\ell[f_0]\|^2=\|\ell[f_0]-\re\lambda f_0\|^2+|\im \lambda|^2\|f_0\|^2\ge 0.\nonumber
\]
Therefore, $J_\lambda$ is positive on $\gD_{\min}$ if $\lambda\in \C_+$.

Let $f_\lambda\in \cN_\lambda$. Then, by \eqref{eq:green}, we get
\[
 J_\lambda(f_0, f_\lambda)=|\lambda|^2(f_0,f_\lambda)-2\bar{\lambda}\, \re\, \lambda (f_0,f_\lambda)+\bar{\lambda}^2(f_0,f_\lambda)=0,
\]
and similarly
\[
 J_\lambda(f_0, f_{\overline{\lambda}})=|\lambda|^2(f_0,f_{\bar{\lambda}})-2\lambda\, \re\, \lambda (f_0,f_{\bar{\lambda}})+\lambda^2(f_0,f_{\bar{\lambda}})=0.
\]
This completes the proof.
\end{proof}
Consider the following functions
\be\label{eq:M_pm}
M_\pm(\lambda)=-\left(\begin{array}{cc}
\re (\lambda m_\pm(\lambda)) &  \re\, \lambda\, \overline{m_\pm(\lambda)}\\
\re\, \lambda\, m_\pm(\lambda) & \re (\lambda m_\pm(\lambda))
\end{array}\right),\quad \lambda\in\C_+.
\ee 
\begin{lemma}\label{lem:dmax}
Let $m_\pm$  and $M_\pm$ be defined by  \eqref{eq:weyl_s} and \eqref{eq:M_pm}, respectively. The form $J_\lambda$ is nonnegative on $\gD_{\max}$ if and only if  both matrices $M_+(\lambda)$ and $M_-(\lambda)$ are nonnegative.
\end{lemma}
\begin{proof}
Firstly,  by \eqref{eq:defect_L}, each $f_\lambda\in \cN_{\lambda}(L_{\min})$ admits the representation 
\[
f_\lambda(x)=c_+\psi_+(x,\lambda)\chi_+(x) + c_-\psi_-(x,\lambda)\chi_-(x), \quad c_\pm\in\C. 
\]
Next observe that
\[
 J_\lambda(f_\lambda)=|c_+|^2 J_\lambda(\psi_+(\lambda)\chi_+)+|c_-|^2 J_\lambda(\psi_-(\lambda)\chi_-).
\]
By \eqref{eq:weyl_s} and \eqref{eq:green}, we get 
\[
\gt[\psi_\pm(\lambda)\chi_\pm]=m_\pm(\lambda) + \overline{\lambda}\frac{\im\, m_\pm(\lambda)}{\im\, \lambda} = \frac{\im (\lambda\, m_\pm(\lambda))}{\im\, \lambda},
\] 
and hence
\be\label{eq:t_pm} 
\gt[f_\lambda]=|c_+|^2\frac{\im (\lambda\, m_+(\lambda))}{\im\, \lambda}+|c_-|^2\frac{\im (\lambda\, m_-(\lambda))}{\im\,\lambda},\quad \lambda\in\C_+.
\ee
Next, we compute
\begin{align}
 J_\lambda(\psi_\pm(\lambda)\chi_\pm)&=2|\lambda|^2\frac{\im\, m_\pm(\lambda)}{\im\, \lambda}-2\re\, \lambda\frac{\im (\lambda m_\pm(\lambda))}{\im\, \lambda}\nn\\
&=2(\im\, \lambda\, \im\, m_\pm(\lambda)-\re\, \lambda\, \re\, m_\pm(\lambda))=-2\re (\lambda\, m_\pm(\lambda)),\label{eq:01}
\end{align}
and
\begin{align}
 J_\lambda(\psi_\pm(\bar{\lambda})\chi_\pm)&=2|\lambda|^2\frac{\im\, m_\pm(\bar{\lambda})}{\im\, \bar{\lambda}}-2\re\, \lambda\frac{\im (\bar{\lambda}\, m_\pm(\bar{\lambda}))}{\im\, \bar{\lambda}}\nn\\
&=2(\im\, \lambda\, \im\, m_\pm(\lambda)-\re\, \lambda\, \re\, m_\pm(\lambda))=-2\re (\lambda m_\pm(\lambda)).\label{eq:02}
\end{align}
Finally, one easily obtains  
\be\label{eq:03}
 J_\lambda(\psi_\pm(\lambda),\psi_\pm(\bar{\lambda}))=-2\re\, \lambda\, m_\pm(\lambda).
\ee
Therefore, using \eqref{eq:t_pm}--\eqref{eq:03} and setting 
\be\label{eq:f_pm}
f_\pm=c_1^\pm\psi_\pm(x,\lambda)\chi_\pm(x)+c_1^\pm\psi_\pm(x,\bar{\lambda})\chi_\pm(x),
\ee
 we get
\be\label{eq:Jf_pm}
J_\lambda(f_\pm)=\big(2M_\pm(\lambda)C_\pm,C_\pm\big),
\ee
where $C_\pm:={\rm col}(c_1^\pm,c_2^\pm)\in\C^2$. 
Hence, $J_\lambda$ is nonnegative on $\cN_\lambda\dotplus\cN_{\bar{\lambda}}$ precisely when both $M_+(\lambda)$ and $M_-(\lambda)$ are nonnegative. Lemma \ref{lem:j_lam} completes the proof.
\end{proof}
\begin{corollary}[Everitt]
There is $K>0$ such that \eqref{eq:volkmer} is satisfied for all $f\in \gD_{\max}$  if and only if 
\be\label{eq:ev_mf}
-\im(\lambda^2m_\pm(\lambda))\ge 0\quad \text{for all} \quad\lambda\in\C_+\quad \text{with}\quad \frac{|\re\lambda|}{|\lambda|}=\frac{1}{K}.
\ee
\end{corollary}
\begin{proof}
Straightforward calculations yield
\[
\det M_\pm(\lambda)=-\im\, m_\pm(\lambda)\, \im(\lambda^2m_\pm(\lambda)).
\]
Lemmas \ref{lem:j_lam} and \ref{lem:dmax} complete the proof.
\end{proof}
Now we are ready to complete the proof of Theorem \ref{th:crit_mf}.

\begin{proof}[Proof of Theorem \ref{th:crit_mf}]
Firstly, observe that $f\in \dom(A)$ if and only if $f\in \gD_{\max}$ and satisfies the boundary conditions \eqref{eq:bc}. 
Further, note that $f\in \gD_{\max}$ admits the representation $f=f_0+f_+ + f_-$, where $f_0\in\gD_{\min}$ and $f_\pm$ are given by \eqref{eq:f_pm}. Hence, by \eqref{eq:weyl_s}, $f\in\dom(A)$ precisely when
\[
\begin{cases}
-c_1^+m_+(\lambda)-c_2^+\overline{m_+(\lambda)}=c_1^-m_-(\lambda)+c_2^-\overline{m_-(\lambda)}\\
-c_1^+ - c_2^+ = c_1^- + c_2^-
\end{cases}
\]
or, equivalently, if $C_\pm=\rm{col}(c_1^\pm,c_2^\pm)\in\C^2$ satisfy
\be\label{eq:d_pmA}
-D_+(\lambda)C_+=D_-(\lambda)C_-,\quad 
D_\pm(\lambda):=\left(\begin{array}{cc}
m_\pm(\lambda) &  \overline{m_\pm(\lambda)}\\
1 & 1
\end{array}\right).
\ee

Notice that $\det D_\pm(\lambda)=2\I\, \im\, m_\pm(\lambda)\neq 0$ if $\lambda\in \C_+$, and hence 
\be\label{eq:d_inv}
D_\pm^{-1}(\lambda)=\frac{1}{m_\pm(\lambda)-{m_\pm(\overline\lambda)}}\left(\begin{array}{cc}
1 &  -{m_\pm(\overline\lambda)}\\
-1 & m_\pm(\lambda)
\end{array}\right),\quad \lambda\in\C_+.
\ee
Therefore, using \eqref{eq:Jf_pm} and \eqref{eq:d_pmA}, we obtain for $f\in\dom(A)$ 
\begin{eqnarray*}
&J_\lambda(f)=J_\lambda(f_+) + J_\lambda(f_-)=2\big(M_+(\lambda)C_+,C_+\big)+2\big(M_-(\lambda)C_-,C_-\big)\\
&=2\big((M_+(\lambda)+D_+(\lambda)^*D_-(\lambda)^{-*}M_-(\lambda)D_-(\lambda)^{-1}D_+(\lambda))C_+,C_+\big).
\end{eqnarray*}
Hence $J_\lambda$ is nonnegative for all $f\in\dom(A)$ if and only if the matrix 
\[
D_+(\lambda)^{-*}M_+(\lambda)D_+(\lambda)^{-1}+D_-(\lambda)^{-*}M_-(\lambda)D_-(\lambda)^{-1}
\]
is nonnegative. Notice that $M_A(\lambda)=\tilde{M}_+(\lambda)+\tilde{M}_-(\lambda)$, where 
\[
\tilde{M}_\pm(\lambda):=4D_\pm(\lambda)^{-*}M_\pm(\lambda)D_\pm(\lambda)^{-1}
=\frac{1}{\im m_\pm(\lambda)}\left(\begin{array}{cc}
\im\lambda &  -\im(\lambda m_\pm(\lambda))\\
-\im(\lambda m_\pm(\lambda)) & \im\lambda |m_\pm(\lambda)|^2
\end{array}\right).
\]
Hence Lemma \ref{lem:3.1} completes the proof.
\end{proof}

\section{Volkmer's inequality: the second criterion}\label{sec:vol_II}
In this section we are going to prove the main result of our paper:
\begin{theorem}\label{th:criterion}
Let $m_+$ and $m_-$ be the $m$-functions \eqref{eq:weyl_s}. Then the Volkmer inequality \eqref{eq:volkmer} is valid if and only if
\be\label{eq:crit_imaginary}
\sup_{y> 0}\frac{\re( m_+(\I y) +  m_-(\I y))}{|m_+(\I y)-m_-(-\I y)|}<+\infty.
\ee
\end{theorem}
Before proving Theorem \ref{th:criterion} we need several preliminary lemmas.
\begin{lemma}\label{lem:det_a}
Let the matrix--function $M_A$ be given by \eqref{eq:ti_m}--\eqref{eq:ma} and define
\be\label{eq:ti_ma2}
\tilde{M}_A(\lambda):=\frac{\det\, M_A(\lambda)}{\im\, m_+(\lambda)\im\, m_-(\lambda)},\quad \lambda\in\C_+.
\ee
 Then for $\lambda=(K+\I)y$, $y>0$, $K\in \R$,
\be\label{eq:ti_ma}
\tilde{M}_A(\lambda)=\big| m_+(\lambda)-\overline{m_-(\lambda)}\big|^2 -4K\, \im \big( m_+(\lambda)\, m_-(\lambda)\big) -4K^2\, \im\, m_+(\lambda) \im\, m_-(\lambda).
\ee
\end{lemma}
\begin{proof}
The proof follows from lengthy but straightforward calculations. Namely, 
set $\lambda=(K+\I) y$. Note that 
\[
M_A(\lambda)=\begin{pmatrix}
M_{11}(\lambda) & M_{12}(\lambda)\\
M_{12}(\lambda) & M_{22}(\lambda)
\end{pmatrix},
\]
where
\begin{align*}
M_{11}(\lambda)=
&\im\, m_+(\lambda)+\im\, m_-(\lambda),\quad M_{22}(\lambda)=\im\, m_+(\lambda)|m_-(\lambda)|^2+\im\, m_-(\lambda)|m_+(\lambda)|^2,\\
&M_{12}(\lambda)= -\im\, m_+(\lambda)\im ((K+\I)m_-(\lambda))-\im m_-(\lambda)\im ((K+\I)m_+(\lambda)).
\end{align*}
Therefore, 
\begin{eqnarray*}
\det M_A(\lambda)&=&(\im\, m_+(\lambda)+\im\, m_-(\lambda))(\im\, m_+(\lambda)|m_-(\lambda)|^2+\im\, m_-(\lambda)|m_+(\lambda)|^2)\nn\\
&-&(\im\, m_+(\lambda)\im ((K+\I)m_-(\lambda))+\im\, m_-(\lambda)\im ((K+\I)m_+(\lambda)))^2.
\end{eqnarray*}

Further, note that
\begin{eqnarray*}
&|m_\pm(\lambda)|^2-(\im((K+\I)m_\pm(\lambda)))^2=(\re\, m_\pm(\lambda))+(\im \, m_\pm(\lambda))^2
-K^2(\im\, m_\pm(\lambda))^2-(\re\, m_\pm(\lambda))^2\\
& -2K\im\, m_\pm(\lambda)\re\, m_\pm(\lambda)=\im\, m_\pm(\lambda)\big[ (1-K^2) \im\, m_\pm(\lambda) -2K \re\, m_\pm(\lambda)\big],
\end{eqnarray*}
and
\begin{eqnarray*}
&\im ((K+\I)m_+(\lambda))\cdot\im ((K+\I)m_-(\lambda))=(K\im\, m_+(\lambda)+\re\, m_+(\lambda))(K\im\, m_-(\lambda)
+\re\, m_-(\lambda))\\
&=K^2\im\, m_+(\lambda)\im\, m_-(\lambda)+K(\im\, m_+(\lambda)\re\, m_-(\lambda)+\im\, m_-(\lambda)\re\, m_+(\lambda))+\re\, m_+(\lambda)\re\, m_-(\lambda).
\end{eqnarray*}
Therefore,
\begin{eqnarray*}
\tilde{M}_A(\lambda)&=&|m_+(\lambda)|^2+|m_-(\lambda)|^2
+\im\, m_+(\lambda)\big[ (1-K^2) \im\, m_-(\lambda) -2K \re\, m_-(\lambda)\big]\\
&+&\im\, m_-(\lambda)\big[ (1-K^2) \im\, m_+(\lambda) -2K \re\, m_+(\lambda)\big] -2K^2\im\, m_+(\lambda)\im\, m_-(\lambda) \\
&-&2K(\im\, m_+(\lambda)\re\, m_-(\lambda) +\im\, m_-(\lambda)\re\, m_+(\lambda))-2\re\, m_+(\lambda)\re\, m_-(\lambda)\\
&=&(\im\, m_+(\lambda))^2+(\im\, m_-(\lambda))^2+2(1-2K^2)\im\, m_+(\lambda) \im\, m_-(\lambda)\\
&-&4K(\im\, m_+(\lambda)\re\, m_-(\lambda)+\im\, m_-(\lambda)\re\, m_+(\lambda))\\
& +&(\re\, m_+(\lambda))^2+(\re\, m_-(\lambda))^2-2\re\, m_+(\lambda)\re\, m_-(\lambda)\\
&=&(\im\, m_+(\lambda)+\im\, m_-(\lambda))^2 +(\re\, m_+(\lambda)-\re\, m_-(\lambda))^2 \\
&-& 4K^2\im\, m_+(\lambda) \im\, m_-(\lambda)-4K(\im\, m_+(\lambda)\re\, m_-(\lambda)+\im\, m_-(\lambda)\re\, m_+(\lambda))\\
&=&\big| m_+(\lambda)-\overline{m_-(\lambda)}\big|^2 -4K \im \big( m_+(\lambda)\, m_-(\lambda)\big) -4K^2\im\, m_+(\lambda) \im\, m_-(\lambda).
\end{eqnarray*}
\end{proof}
\begin{corollary}\label{cor:2side}
Let $\lambda=(K+\I)y\in\C_+$ and $\tilde{M}_A$ be given by \eqref{eq:ti_ma}. Then 
\begin{align}
(1-K^2)\big| m_+(\lambda)-\overline{m_-(\lambda)}\big|^2 &-4K \im \big( m_+(\lambda)\, m_-(\lambda)\big)\le \tilde{M}_A(\lambda)\label{eq:2side_A}\\
&\le \big| m_+(\lambda)-\overline{m_-(\lambda)}\big|^2 -4K \im \big( m_+(\lambda)\, m_-(\lambda)\big).\label{eq:2side_B}
\end{align}
\end{corollary}
\begin{proof}
Inequality \eqref{eq:2side_B} clearly follows from \eqref{eq:ti_ma}. To prove \eqref{eq:2side_A} it suffices to note that
\[
\big| m_+(\lambda)-\overline{m_-(\lambda)}\big|^2\ge (\im\, m_+(\lambda)+\im \, m_-(\lambda))^2\ge 4\im\, m_+(\lambda) \im\, m_-(\lambda).
\]
\end{proof}

\begin{corollary}\label{cor:neg_k}
If $\lambda=(K+\I) y\in\C_+$ and $K\in [-1,0]$, then $\tilde{M}_A(\lambda)\ge 0$.
\end{corollary}
\begin{proof}
Using the representation \eqref{eq:m_repr} of $m_\pm$, we get $\im(m_+(\lambda)m_-(\lambda))>0$ if $\lambda=(K+\I)y$ with $K\le 0$ and $y>0$. Inequality \eqref{eq:2side_A} completes the proof.
\end{proof}
Therefore, we can restrict our considerations to the first quarter.

Now let us demonstrate that the condition stronger than \eqref{eq:crit_imaginary} is sufficient for the validity of \eqref{eq:volkmer}. Set 
\[
S_k:=\{\lambda=(K+\I) y: y>0, \ K\in (0,k)\}.
\] 
\begin{lemma}\label{lem:suf_1}
If there is $k>0$ such that 
\be\label{eq:suf01}
\sup_{\lambda\in S_k}\frac{|\im \big( m_+(\lambda)\, m_-(\lambda)\big)|}{\big| m_+(\lambda)-\overline{m_-(\lambda)}\big|^2}=C_0<\infty,
\ee
then  Volkmer's inequality \eqref{eq:volkmer} is valid.
\end{lemma}
\begin{proof}
Set  $k_0:=\sqrt{C_0^2+1}-C_0>0$. Note that $k_0$ satisfies 
\[
\frac{1-k_0^2}{2k_0}=C_0.
\] 
Hence \eqref{eq:2side_A} implies that  $\tilde{M}_A(\lambda)\ge 0$  for all $\lambda=(K+\I)y$ with $y>0$ and $K\in(0,\tilde{k})$, where $\tilde{k}=\min\{k,k_0\}$. By \eqref{eq:ti_ma2}, Corollary \ref{cor:neg_k}  and Theorem \ref{th:crit_mf}, Volkmer's inequality holds true with the constant $K_0:=\sqrt{1+\frac{1}{\tilde{k}^2}}$.
\end{proof}

As an immediate corollary of Lemma \ref{lem:suf_1} we obtain that the left-hand side in \eqref{eq:suf01} is unbounded in any sector $S_k$ if \eqref{eq:volkmer} is not valid.
\begin{corollary}\label{cor:nes_2}
If Volkmer's inequality \eqref{eq:volkmer} is not valid, then there is a sequence $\{\lambda_n\}_1^\infty\subset\C_+$ such that
\be\label{eq:seq_z}
\lambda_n=(k_n+\I) y_n, \quad y_n \to +\infty,\quad k_n\to+0,
\ee
and 
\be\label{eq:seq_ineq}
\frac{|\im \big( m_+(\lambda_n)\, m_-(\lambda_n)\big)|}{\big| m_+(\lambda_n)-\overline{m_-(\lambda_n)}\big|^2}\to +\infty.
\ee
\end{corollary}
\begin{proof}
By Theorem \ref{th:crit_mf} and \eqref{eq:ti_ma2}, if Vollkmer's inequality is not valid, then for any sequence $k_n\downarrow 0$ there is a sequence  $y_n>0$ such that 
\[
\tilde{M}_A(\lambda_n)<0,\quad \lambda_n=(k_n+\I)y_n,\quad (n\in\N).
\]
Therefore, by \eqref{eq:2side_A}, 
\[
\frac{|\im \big( m_+(\lambda_n)\, m_-(\lambda_n)\big)|}{\big| m_+(\lambda_n)-\overline{m_-(\lambda_n)}\big|^2}\ge \frac{1-2k_n^2}{4k_n}\to +\infty\quad \text{as}\quad k_n\to+0.
\]
Finally, note that $\lambda_n$ accumulates at $\infty$. Indeed, both $m_+$ and $m_-$ have a simple pole at $\lambda=0$ and hence $\lambda_n$ cannot accumulate at $0$. Moreover, $\im (m_+(\I y)-\overline{m_-(\I y)})=\im m_+(\I y)+\im m_-(\I y)>0$ for any finite $y>0$.
\end{proof}
In the next corollary we shall show that one can choose a sequence $\lambda_n$ with the properties \eqref{eq:seq_z} and such that it satisfies a condition stronger than \eqref{eq:seq_ineq}.
\begin{corollary}\label{cor:nes_2B}
If Volkmer's inequality \eqref{eq:volkmer} is not valid, then there is a sequence $\{\lambda_n\}_1^\infty\subset\C_+$ satisfying \eqref{eq:seq_z} and such that
\be\label{eq:seq_ineqB}
\lim_{n\to+\infty}\frac{\re \big( m_+(\lambda_n)+ m_-(\lambda_n)\big)}{\big| m_+(\lambda_n)-\overline{m_-(\lambda_n)}\big|}=+\infty.
\ee
\end{corollary}
\begin{proof}
By Corollary \ref{cor:nes_2}, there is a sequence $\lambda_n$ with the properties \eqref{eq:seq_z}--\eqref{eq:seq_ineq} if \eqref{eq:volkmer} is not valid. Therefore, at least one of the following sequences
\[
\frac{\im\, m_+(\lambda_n)\re\, m_-(\lambda_n)}{| m_+(\lambda_n)-\overline{m_-(\lambda_n)}|^2},\quad \frac{\im\, m_-(\lambda_n)\re\, m_+(\lambda_n)}{| m_+(\lambda_n)-\overline{m_-(\lambda_n)}|^2}
,\quad n\in\N,
\]
is unbounded. Assume that the second sequence is unbounded. Noting that
\[
\frac{\im\, m_\pm(\lambda)}{|m_+(\lambda)-\overline{m_-(\lambda)}|}\le \frac{\im\, m_\pm(\lambda)}{\im (m_+(\lambda)+m_-(\lambda))}<1,\quad \lambda\in\C_+,
\]
we get
\be\label{eq:5.08B}
\sup_{n}\frac{\re \, m_+(\lambda_n)}{\big| m_+(\lambda_n)-\overline{m_-(\lambda_n)}\big|}=\infty.
\ee
Without loss of generality we can assume that 
\[
\frac{|\re \, m_+(\lambda_n)|}{\big| m_+(\lambda_n)-\overline{m_-(\lambda_n)}\big|}\to +\infty\quad \text{as}\quad n\to \infty.
\]
The latter, in particular, implies
\be\label{eq:5.09}
\frac{\re\, m_+(\lambda_{n})-\re\, m_-(\lambda_{n})}{\re\, m_+(\lambda_{n})}\to 0,\quad \frac{\im\, m_\pm(\lambda_{n})}{\re\, m_+(\lambda_{n})}\to 0,\quad \text{as}\quad n\to\infty,
\ee
Again, using the representation  \eqref{eq:m_repr} and the assumption that $\lambda_n$ is asymptotically imaginary, we get $m_\pm(\lambda_n)\to 0$ as $n\to\infty$, and hence the first relation in \eqref{eq:5.09} yields 
\be\label{eq:5.09C}
\re\, m_+(\lambda_{n})=\re\, m_-(\lambda_{n})(1+o(1)),\quad n\to\infty.
\ee
On the other hand, the second relation in \eqref{eq:5.09} and \eqref{eq:5.08} implies that $\re\, m_+(\lambda_{n})$ and hence $\re\, m_-(\lambda_{n})$ are positive for sufficiently large $n\in\N$. 
Combining this fact and the last relation with \eqref{eq:5.08B}, we arrive at \eqref{eq:seq_ineqB}.
\end{proof}

The most difficult part of the proof of Theorem \ref{th:criterion} is to prove the statement converse to Corollary \ref{cor:nes_2}. The two-sided estimate \eqref{eq:2side_A}--\eqref{eq:2side_B} shows that for $\tilde{M}_A$ to be positive in some sector $S_k$ it is necessary that the left-hand side in \eqref{eq:suf01} is less than $\frac{1}{4K}$ if $\lambda=(K+\I) y$, $y>0$. But \eqref{eq:2side_A} allows some growth in $K$ for the left-hand side in \eqref{eq:suf01}. Namely, if it is less than $\frac{1-K^2}{4K}$ for all $\lambda=(K+\I) y$, then $\tilde{M}_A$ is nonnegative on the corresponding ray. So, to prove the converse implication we need to show that for any $K>0$ there is $y_K>0$ such that the fraction at the left-hand side in \eqref{eq:suf01} is greater than $\frac{1}{4K}$. 

\begin{lemma}\label{lem:5.4}
If there is a sequence $\{\lambda_n\}\subset\C_+$ satisfying \eqref{eq:seq_z} and such that \eqref{eq:seq_ineqB} holds, then Volkmer's inequality \eqref{eq:volkmer} is not valid.
\end{lemma}
\begin{proof}

We divide the proof in two steps.

$(i)$ Let the sequence $\{\lambda_n\}_1^\infty$ satisfy \eqref{eq:seq_z} and \eqref{eq:seq_ineqB}.
The latter means that 
\[
\text{either}\quad \frac{ \re\, m_+(\lambda_n)}{\big| m_+(\lambda_n)-\overline{m_-(\lambda_n)}\big|}\to \infty\quad \text{or}\quad \frac{\re\, m_-(\lambda_n)}{\big| m_+(\lambda_n)-\overline{m_-(\lambda_n)}\big|}\to \infty.
\]
Without loss of generality we assume that the first sequence tends to infinity. Similar to the proof of Corollary \ref{cor:nes_2B}, we arrive at the following relations 
\be\label{eq:5.02}
\im\, m_\pm(\lambda_n)=o(\re\, m_\pm(\lambda_n)),\quad \re\, m_-(\lambda_n)=\re\, m_+(\lambda_n)(1+o(1)),\quad n\to \infty.
\ee
By Lemma \ref{lem:Bennewitz}, the first equality in \eqref{eq:5.02} yields 
\[
S_0^\pm(x)\equiv 1,\quad x\in (0,1),
\]
where the functions $S_0^\pm$ are defined by \eqref{Parfenov}.

$(ii)$ Next we shall use some arguments from \cite{Ben87} (see also Appendix). Set 
\[
\rho_n:=|\lambda_n|,\quad (n\in N),
\]
and define 
\be
R_\pm(x):=\int_0^x |r(\pm t)|dt,\quad x\in (0,1). 
\ee
Note that the function $R_\pm$ is strictly increasing and absolutely continuous on $(0,1)$. Denote by $R_\pm^{-1}$ its inverse and define $f_\pm$ as the inverse of $1/(R_\pm^{-1}(x))^2$. Note that $f_\pm(t)\downarrow 0$ as $t\uparrow +\infty$. 
Next,  consider the following sequence of functions
\[
P_n^\pm(x):=\frac{R_\pm(u_nx)}{R_\pm(u_n)},\quad u_n:=\frac{1}{\rho_nf_\pm(\rho_n)}, \quad (n\in\N).
\] 
Note that $u_n\to 0$ since $\rho_n\to+\infty$ (see Appendix  \ref{ss:a02}). Hence $P_n^\pm$ is increasing and maps $[0,1]$ onto itself. By the Helly theorem we may choose a subsequence of $\rho_n$ so that $P_n^\pm$ converges, pointwise and boundedly along this subsequence, to some increasing function $P_\infty^\pm$. Without loss of generality we can assume that $P_n^\pm$ converges to $P_\infty^\pm$ along the sequence $\rho_n$. 

By \eqref{eq:seq_z} and \eqref{eq:seq_ineqB}, $\lim_{n}\frac{\lambda_n}{\rho_n}= \I$ and $\frac{m_\pm(\lambda_n)}{f_\pm(\rho_n)}$ is asymptotically real. Therefore (see Appendix \ref{ss:a02}), $\frac{m_\pm(\lambda_n)}{f_\pm(\rho_n)}$ is asymptotically in the Weyl circle of the system \eqref{eq:a04} with $P=P_\infty^\pm$. Therefore,  the Weyl circle at $\lambda=\I$ contains a real point. The latter is possible if and only if $P_\infty^\pm(x)\equiv a_\pm$ on $(0,1)$, i.e., $dP_\infty^\pm(x)=a_\pm\delta(x)$ (see Example \ref{ex:a01}). Clearly, $a_\pm\in(0,1]$. 

Therefore (see Appendix \ref{ss:a02} and Example \ref{ex:a01}),  
 \[
 \frac{m_\pm(\lambda_n)}{f_\pm(\rho_n)} \quad \text{is asymptotically in the circle}\quad C_{\frac{1}{2}}(a_\pm+\frac{\I}{2}),
 \]
and hence 
\be\label{eq:5.3a}
 m_\pm(\lambda_n)=a_\pm f_\pm(\rho_n)(1+o(1)),\quad n\to \infty.
\ee
Moreover, \eqref{eq:5.02} implies 
 \be\label{eq:5.03}
a_- f_-(\rho_n)=a_+f_+(\rho_n)(1+o(1)),\quad n\to\infty.
 \ee
Furthermore, $m_\pm(\mu\rho_n)/a_\pm f_\pm(\rho_n)$ is asymptotically in the circle 
\be\label{eq:5.03B}
C_{1/2a_\pm|\im \mu|}\big(1+\frac{\I}{2a_\pm|\im \mu|}\big)=\big\{z:\ \big|z-1-\frac{\I}{2a_\pm|\im \mu|}\big|=\frac{1}{2a_\pm|\im \mu|}\big\}.
\ee
 The latter holds uniformly for $\mu $ in any compact set not intersecting $\R$.
 
(iii) Now fix $K>0$ and let $\mu\in \Gamma_K=\{\lambda:\ \lambda=(K+\I)y,\ y>0\}$. 
 Consider the sequence
 \be\label{eq:5.04}
 M_K(n):=\frac{|\im (m_+(\mu\rho_n)m_-(\mu\rho_n))|}{|m_+(\mu\rho_n)-\overline{m_-(\mu\rho_n)}|^2}.
 \ee
 Notice that for any two points $z_+$ and $z_-$ in the circle $B_{\rho}(1+\I\rho)$ the following inequality
 \be\label{eq:5.05}
 \frac{\im (z_+z_-)}{|z_+-\overline{z_-}|^2}\ge \frac{1-\rho}{8\rho}
 \ee
holds true for a sufficiently small $\rho>0$. Indeed, if $z_+$ and $z_-$ approach $\lambda=1$, the left-hand side in \eqref{eq:5.05} tends to infinity and hence \eqref{eq:5.05} is satisfied in some neighborhood of $\lambda=1$. Assume that $z_\pm\neq 1$, i.e., $\im\, z_\pm>0$. Then
\[
\im (z_+z_-)\ge \min\{\re\, z_+,\re\, z_-\}(\im\, z_++\im\, z_-)\ge (1-\rho)(\im\, z_++\im\, z_-).
\]
On the other hand, we get
\[
\frac{|z_+-\overline{z_-}|^2}{\im\, z_+ + \im\, z_-}
=\frac{(\re\, z_+ - \re\, z_-)^2}{\im\, z_+ +\im\, z_-}+\im\, z_+ +\im\, z_-
\le 2\frac{(\re\, z_+ -1)^2+ (\re\, z_- -1)^2}{\im\, z_+ +\im\, z_-}+4\rho
\]
However, for any $z_\pm \in C_{\rho}(1+\I\rho)\setminus\{1\}$ 
\[
(\re z_\pm -1)^2= \rho^2-(\im\, z_\pm-\rho)^2=\im\, z_\pm(2\rho-\im\, z_\pm)\le 2\rho\im\, z_\pm,
\]
and hence we finally get
\[
\frac{|z_+-\overline{z_-}|^2}{\im\, z_+ +\im\, z_-}\le 8\rho,
\]
which proves \eqref{eq:5.05}.

Therefore, for $n$ large enough we get 
\be\label{eq:5.06}
\frac{|\im (m_+(\mu\rho_n)m_-(\mu\rho_n))|}{|m_+(\mu\rho_n)-\overline{m_-(\mu\rho_n)}|^2}\ge 
\frac{2a\, \im\, \mu-1}{8},\quad a:=\min\{a_+,a_-\}>0. 
\ee
Therefore, choosing any $y_K>\frac{1}{aK}+\frac{1}{2a}$ and setting $\mu_K:=(K+\I)y_K$, we see that the left-hand side in \eqref{eq:5.06} is greater than $\frac{1}{4K}$. Thus, by \eqref{eq:2side_B}, we conclude that $\tilde{M}_A(\mu_K\rho_n)$ becomes negative for  $n\in\N$ large enough. Using \eqref{eq:ti_ma2}, we conclude that for any $K>0$ there is $y_K>0$ such that $\det M_A((K+\I)y_K)<0$. Theorem \ref{th:crit_mf} completes the proof.
\end{proof}
Now we are ready to prove the main theorem.

\begin{proof}[Proof of Theorem \ref{th:criterion}]
The proof is similar to the proof of sufficiency of Theorem \ref{th:ak_crit+}. 

{\em Sufficiency.} Assume the converse, that is \eqref{eq:volkmer} is not valid. Let us show that in this case the supremum in \eqref{eq:crit_imaginary} equals $+\infty$. 

Firstly, by Corollary \ref{cor:nes_2B},  there is a sequence $\{\lambda_n\}\in\C_+$ satisfying \eqref{eq:seq_z} and \eqref{eq:seq_ineqB}. Then, arguing as in the proof of Corollary \ref{cor:nes_2B}, one shows that \eqref{eq:5.02} holds true. 
Next, by \eqref{eq:5.07} and \eqref{eq:5.08}, 
for $k_n<1/2$ and $n$ large enough we have
\be\label{eq:5.10}
|\im\, m_\pm(\lambda_n) - \im\, m_\pm(\I y_n)| \le 2k_n\, \im\, m_\pm(\lambda_n),\quad
|\re\, m_\pm(\lambda_n) - \re\, m_\pm(\I y_n)| \le 3k_n\, \im\, m_\pm(\lambda_n)
\ee
The latter immediately implies 
\[
|m_\pm(\lambda_n)-m_\pm(\I y_n)|\le 4k_n\, \im\, m_\pm(\lambda_n).
\]
Using the first relation in \eqref{eq:5.02}, we obtain
\begin{eqnarray*}
&\re\, m_\pm(\I y_n)=\re\, m_\pm(\lambda_n)(1+o(1)),\\
& |m_+(\I y_n)-m_-(-\I y_n)|=|m_+(\lambda_n)-\overline{m_-( \lambda_n)}|(1+o(1)),\quad n\to\infty.
\end{eqnarray*}
Combining this with \eqref{eq:seq_ineqB} and using the second relation in \eqref{eq:5.02},  we immediately get 
\be\label{eq:5.14}
\lim_{n\to+\infty}\frac{\re( m_+(\I y_n) +  m_-(\I y_n))}{|m_+(\I y_n)-m_-(-\I y_n)|}=+\infty,
\ee
which proves sufficiency. 

{\em Necessity.} Assume that \eqref{eq:crit_imaginary} is not satisfied. To prove the claim it suffices to show that there is a sequence $\{\lambda_n\}\subset\C_+$ satisfying \eqref{eq:seq_z} and \eqref{eq:seq_ineqB}. 

Firstly,  there is $y_n\to+\infty$ such that \eqref{eq:5.14} holds. 
The latter implies \eqref{eq:5.02} with $\I y_n$ instead of $\lambda_n$. 

Choose an arbitrary sequence $K_n\to +0$ and set $\lambda_n:=(K_n+\I)y_n$, $n\in\N$. Again, using \eqref{eq:5.02} and \eqref{eq:5.07}--\eqref{eq:5.08}, one arrives at the estimates \eqref{eq:5.10}. The rest of the proof is analogous to the proof of sufficiency and we left it to the reader.   
\end{proof}
We complete this section with the following remark.
\begin{remark}
The proof of Theorem \ref{th:crit_mf} is based on operator--theoretic ideas from \cite{ee91, ez78} and, clearly, the analysis extends to general Sturm--Liouville differential expressions and even to the abstract settings. However, the proof of Theorem \ref{th:criterion} exploits two specific features of \eqref{eq:volkmer}: (i) the integral represenation \eqref{eq:m_repr} of $m$-coefficients $m_+$ and $m_-$, (ii) connections between the asymptotic behavior of $m_\pm$ and the behavior of $r$ at $\pm 0$. Therefore, Theorem \ref{th:criterion} can be extended at least to the case of general regular Sturm--Liouville expressions under the assumption that $m_+$ and $m_-$ admit the representation \eqref{eq:m_repr}--\eqref{eq:m_repr2}.   
\end{remark}

\section{Necessary and sufficient conditions for the validity of Volkmer's inequality in terms of weights} \label{sec:cond}

\subsection{The case of odd $r$}\label{ss:odd}
Clearly, the condition \eqref{eq:ev_mf} implies \eqref{eq:m_cond}. Indeed, \eqref{eq:ev_mf} is satisfied precisely if \eqref{eq:volkmer} holds for all $f\in\gD_{\max}$, however, \eqref{eq:m_cond} gives a criterion for the validity of \eqref{eq:volkmer} on a smaller domain $\dom(A)$. The converse implication is not true in general (cf. Section \ref{ss:scaling} below). However, in  exceptional cases it is indeed true. In particular, the next result shows that it is true  for odd $r$.
 
\begin{lemma}\label{cor:3.2}
Assume that $r\in L^1(-1,1)$ is odd and $xr(x)>0$ a.e. on $(-1,1)$. Let also $m_+$ be the $m$-function defined by \eqref{eq:weyl_s}. Then the following are equivalent:
\begin{enumerate}
\item[(i)] Volkmer's inequality \eqref{eq:volkmer} is valid,
\item[(ii)] the HELP inequality \eqref{eq:help} is valid,
\item[(iii)] there is $K>0$ such that $m_+$ satisfies \eqref{eq:ev_mf},
\item[(iv)] 
\be
\sup_{y>0}\frac{\re\, m_+(\I y)}{\im \, m_+(\I y)}<+\infty,
\ee
\item[(v)] the function $r$ satisfies Bennewitz's condition \eqref{Parfenov}.
\end{enumerate}
\end{lemma}
\begin{proof}
We only need to establish the equivalence $(i)\Leftrightarrow (ii)$. Since $r$ is odd, we clearly get $m_+(\lambda)=m_-(\lambda)$. Therefore (see \eqref{eq:ti_m}, \eqref{eq:ma}), $M_A(\lambda)=2\tilde{M}_+(\lambda)$. Noting that $\det\, \tilde{M}_+(\lambda)=-\frac{\im(\lambda^2\, m_+(\lambda))}{\im\, m_+(\lambda)}$, Theorems \ref{th:ev71} and \ref{th:crit_mf} complete the proof.
\end{proof}

\subsection{Nonodd weights}\label{ss:nonodd}
In the case of non-odd weights, the validity of \eqref{eq:volkmer} is an open problem. Note that Theorem \ref{th:crit_mf} and Theorem \ref{th:criterion} provide two critera for the validity of Volkmer's inequality in terms of $m$-functions. Our next aim is to apply Bennewitz's criterion in order to get simple sufficient conditions for the validity of \eqref{eq:volkmer} in terms of the weight $r$.

\begin{definition}
Let $r\in L^1(-1,1)$ and $xr(x)>0$ a.e. on $(-1,1)$. We say that the function $r$ satisfies Bennewit's condition at $+0$ ($-0$) if there is $t\in (0,1)$ such that \eqref{Parfenov} is satisfied with $r(x)$ ($|r(-x)|$).
\end{definition}

\begin{lemma}\label{lem:ben+-}
Let $r\in L^1(-1,1)$ and $xr(x)>0$ a.e. on $(-1,1)$.  If $r$ satisfies Bennewitz's condition either at $+0$ or at $-0$, then the inequality \eqref{eq:volkmer} is valid.
\end{lemma}
\begin{proof}
The proof immediately follows from the proof of Lemma \ref{lem:5.4}. Namely, assume that \eqref{eq:volkmer} is not valid. Then there is a sequence $\lambda_n$ satisfying \eqref{eq:seq_z} and \eqref{eq:5.02}. Lemma \ref{lem:Bennewitz} implies that $r$ does not satisfy Bennewitz's condition at both $+0$ and $-0$.
\end{proof}

\begin{remark}
Note that one can easily get sufficient conditions for the validity of \eqref{eq:volkmer} by applying Theorem \ref{lem:vol} and known sufficient conditions for the Riesz basis property (see, e.g., \cite{BF2, Par03, Vol_96} and references therein). In particular, Lemma \ref{lem:ben+-} follows from Theorem \ref{lem:vol} and Corollary 4 from \cite{Par03}.
\end{remark}

\subsection{Nonsymmetric scaling}\label{ss:scaling}

In this subsection we concentrate our attention on a particular class of non-odd weights 
\be\label{eq:tilde_r}
\tilde{r}(x):=\begin{cases}
r(x),& x\in(0,1),\\
-a^2r(-ax),&x\in (-1/a,0),
\end{cases}
\ee
where $a>0$ is a positive constant and $r:(0,1)\to \R_+$. Namely we consider \eqref{eq:volkmer} with $\tilde{r}$ in place of $r$ and on the interval $[-1/a,1]$ instead of $[-1,1]$. 
\begin{lemma}\label{lem:scaling}
Let $a>0$ and $\tilde{r}$ be given by \eqref{eq:tilde_r}, where $r\in L^1(0,1)$ is positive a.e. on $(0,1)$. If $a\neq 1$, then \eqref{eq:volkmer} is valid with  
\[
K_a:=\left|\frac{1+a}{1-a}\right|. 
\]
Moreover, the constant $K_a$ is the best possible.
\end{lemma}

\begin{proof}
Without loss of generality we can assume that $a>1$. Let $m_+$ and $m_-$ be the corresponding $m$-function. 
Firstly, observe that for weights \eqref{eq:tilde_r} the $m$-functions $m_+$ and $m_-$ are connected by 
\[
m_-(\lambda)=am_+(\lambda),\quad \lambda\in\C_+.
\]
Indeed, let $m$ and $m_{c}$ be the m-functions corresponding to Neumann problems for equations
\[
-(r(x)^{-1}y')'= \lambda\, y,\quad x\in [0,1]
\]
and 
\[
-((a^2r(ax))^{-1}y')'= \lambda\, y,\quad x\in [0,1/a],
\]
respectively. Then the fundamental systems solutions of these equations are connected by 
\[
\tilde{c}(x,\lambda) = c(ax,\lambda),\quad \tilde{s}(x,\lambda)=a\, s(ax,\lambda)
\]
Moreover,Weyl solutions satisfy 
\[
\tilde{\psi}(x,\lambda)=\tilde{s}(x,\lambda) - am(\lambda)\tilde{c}(x,\lambda)=a\psi(ax,\lambda).
\]
Therefore, we get
\be\label{eq:connection}
m_{a}(\lambda)=am(\lambda).
\ee

Further, by \eqref{eq:ma} we obtain
\[
M_A(\lambda)=\frac{1}{\im m_+(\lambda)}\left( \begin{array}{cc}
(1+1/a)\im\, \lambda & -2\im (\lambda m_+(\lambda))\\
 -2\im (\lambda m_+(\lambda)) & (1+a) \im\, \lambda |m_+(\lambda)|^2
\end{array}\right),
\]
and hence
\[
\det M_A(\lambda)=4 \big(\im\, m_+(\lambda)\big)^2\big((1+d) (\im\, \lambda\,|m_+(\lambda)|)^2- (\im (\lambda m_+(\lambda)))^2 \big),
\]
where $d=(1+1/a)(1+a)/4-1>0$. Then, setting $\lambda=(k+\I)y$, after straightforward calculations we obtain
\[
\frac{\det M_A(\lambda)}{4y^2\big(\im\, m_+(\lambda)\big)^2}=\big( \sqrt{d}\re\, m_+(\lambda) - \frac{k}{\sqrt{d}} \im\, m_+(\lambda)\big)^2 + (1+d-k^2 -\frac{k^2}{d})(\im\, m_+(\lambda))^2.
\]
Hence $M_A$ is nonnegative precisely if
\[
k^2\le d.
\]
Therefore, by Theorem \ref{th:crit_mf},  \eqref{eq:volkmer} holds true with $K=\sqrt{\frac{1+d}{d}}=\left|\frac{1+a}{1-a}\right|$ and this constant is the best possible. 
\end{proof}
\begin{remark}
By Lemma \ref{lem:ben+-}, for the validity of \eqref{eq:volkmer} it suffices that $r$ satisfies Bennewitz's condition either at $+0$ or at $-0$. Moreover, it becomes necessary if $r$ is odd. However, Lemma \ref{lem:scaling} shows that for nonodd $r$ Bennewitz's condition is not necessary. 
\end{remark}

\section{On a connection with the linear resolvent growth condition and the Riesz basis property of eigenfunctions}\label{sec:LRG}

Let $r \in L^1(-1,1)$ be real and $xr(x)>0$ for almost all $x\in[-1,1]$. Consider the regular indefinite
Sturm-Liouville eigenvalue problem
\begin{align}\label{evp}
&-f'' = {\lambda} r(x) f
\quad \mbox{on} \quad [-1,1], \\
&f(-1) = f(1) = 0.\nn 
\end{align}
Consider the corresponding operator $H$ in $L^2_{|r|}(-1,1)$ 
\be\label{eq:H}
H:=-\frac{1}{r(x)}\frac{d^2}{dx^2},\quad \dom(H)=\{f\in L^2_{|r|}(-1,1):\ f,f'\in AC[-1,1],\ f(\pm 1)=0,\ r^{-1}f''\in L^2_{|r|}\}.
\ee
The operator $H$ is non self-adjoint in the Hilbert space $L^2_{|r|}(-1,1)$. However, its spectrum is real,  discrete and all eigenvalues are simple. 
There are several necessary and sufficient conditions for $H$ to be similar to a self-adjoint operator, or equivalently, for the eigenfunctions of \eqref{evp} to form a Riesz basis of $L^2_{|r|}(-1,1)$. Notice that Volkmer's Theorem provides a necessary condition. Another necessary condition based on the linear resolvent growth (LRG) condition was obtained in \cite{KarKos} (see also \cite{KKM09}). The main aim of this section is to show that these two necessary conditions are closely connected. 

We need some preliminary notation. Consider the auxiliary spectral problem 
\be\label{eq:sp_r}
-f'' = {\lambda} |r(x)| f,\quad x\in [-1,1];\qquad f(-1) = f(1) = 0.
\ee
With \eqref{eq:sp_r} one naturally associates the $m$-functions $m_+^r$ and $m_-^r$. Namely, let $c^r(x,\lambda)$ and $s^r(x,\lambda)$ be the fundamental system of solutions of equation \eqref{eq:sp_r}, 
\[
c^r(0,\lambda)=(s^r)'(0,\lambda)=1,\quad s^r(0,\lambda)=(c^r)'(0,\lambda)=0.
\]
The corresponding Weyl solutions are defined as follows
\be\label{eq:ws_r}
\psi_\pm^r(x,\lambda)=c^r(x,\lambda)\pm m^r_\pm(\lambda)s^r(x,\lambda),\quad \psi_\pm^r(\pm1,\lambda)=0.
\ee
The functions $m_+^r$ and $m_-^r$ in \eqref{eq:ws_r} are called {\em the Weyl--Titchmarsh $m$-functions} for the problem \eqref{eq:sp_r} subject to the Dirichlet conditions. 

\begin{theorem}[\cite{KarKos}]\label{kostenko}
Let $m_\pm^r$ be the $m$-functions defined by \eqref{eq:ws_r}. If the eigenvalues of \eqref{evp} form a Riesz basis of $L^2_{|r|}(-1,1)$, then
\begin{equation}\label{eq:kostenko}
\sup_{\lambda\in \C_+}\frac{\im\, m_\pm^r(\lambda)}{|m_+^r(\lambda)+m_-^r(-\lambda)|}<+\infty.
\end{equation}
Moreover,  if the operator $H$ satisfies the linear resolvent growth condition, i.e., there is $C>0$ such that
\be\label{lrg}
\|(H-\lambda)^{-1}\|\le \frac{C}{\im \, \lambda},\quad (\lambda\in \C\setminus\R),
\ee
then \eqref{eq:kostenko} is satisfied.
\end{theorem}
\begin{remark}
Note that in \cite{KarKos} Theorem \ref{kostenko} was established for singular indefinite Sturm--Liouville problems under the assumption that the corresponding Sturm--Liouville differential expression is limit point at both singular endpoints. However, the result remains true in a general situation, in particular for regular problems with weight functions having only one turning point. Moreover, it is shown in \cite{Mal} that the conditions \eqref{eq:kostenko} and \eqref{lrg} are equivalent even in an abstract operator theoretic setting. 
\end{remark}
The linear resolvent growth condition as well as Volkmer's inequality are necessary for the Riesz basis property of eigenfunctions of \eqref{evp}. Theorem \ref{th:criterion} enables us to establish a connection between these two necessary conditions.
\begin{theorem}\label{th:connection}
If the operator $H$ defined by \eqref{eq:H} satisfies the linear resolvent growth condition \eqref{lrg}, then the Volkmer inequality \eqref{eq:volkmer} is valid.
\end{theorem}
\begin{proof}
Let $m_\pm^r$ and $m_\pm$ be the $m$-functions defined by \eqref{eq:ws_r} and \eqref{eq:weyl_s}, respectively. Note that the solutions of \eqref{eq:sp_r} and the solutions of \eqref{eq:fs} are  connected as follows 
\[
c(x,\lambda)=(s^r)'(x,\lambda),\quad s(x,\lambda)=-\frac{1}{\lambda}(c^r)'(x,\lambda),
\]
and
\[
\psi_\pm(x,\lambda)=-\frac{1}{\lambda}(\psi_\pm^r)'(x,\lambda).
\]
Therefore, \eqref{eq:ws_r} and \eqref{eq:weyl_s} immediately imply
\be\label{eq:m_conn}
m_\pm(\lambda)=\frac{1}{\lambda}m_\pm^r(\lambda),\quad (\lambda\in \C\setminus\R_+).
\ee
Now, since $H$ satisfies \eqref{lrg}, by Theorem \ref{kostenko} we see that $m_+^r$ and $m_-^r$ satisfy \eqref{eq:kostenko}. The latter in particular implies 
\[
\sup_{y>0}\frac{\im(m_+^r(\I y)+m_-^r(\I y))}{|m_+^r(\I y)+m_-^r(-\I y)|}<\infty.
\]
However, using \eqref{eq:m_conn}, we get that
\[
m_\pm^r(\pm \I y)= \pm \I y \, m_\pm(\pm \I y),\quad (y>0),
\]
and hence
\[
\frac{\im(m_+^r(\I y)+m_-^r(\I y))}{|m_+^r(\I y)+m_-^r(-\I y)|}=\frac{\im(\I y\, m_+(\I y)+\I y\, m_-(\I y))}{|\I y\, m_+(\I y)-\I y\, m_-(-\I y)|}=\frac{\re(m_+(\I y)+m_-(\I y))}{|m_+(\I y)-m_-(-\I y)|},\quad (y>0).
\]
The latter and \eqref{eq:kostenko} yield \eqref{eq:crit_imaginary} and hence, by Theorem \ref{th:criterion}, Volkmer's inequality \eqref{eq:volkmer} is valid.
\end{proof}
It is interesting to note that in the case of odd $r$ both necessary conditions coincide and become also sufficient. This enables us to extend the list of various criteria for the Riesz basis property presented in \cite{BF2}.
\begin{theorem}\label{th:sim_odd}
Let the function $r\in L^1(-1,1)$ be odd and $xr(x)>0$ a.e. on $[-1,1]$.  Let $H$ be the operator \eqref{eq:H} associated with the problem \eqref{evp} and let $m_+^r$ be the $m$-function defined by \eqref{eq:ws_r}. The following statements are equivalent:
\begin{enumerate}
\item[(i)] the eigenvalue problem (\ref{evp}) has the Riesz basis property.
\item[(ii)] the operator $H$ satisfies the LRG condition \eqref{lrg}.
\item[(iii)] the Volkmer inequality \eqref{eq:volkmer} is valid.
\item[(iv)] the function $r$ satisfies Bennewitz's condition \eqref{Parfenov} at $x=0$.
 \item[(v)]
\begin{equation}
\sup_{\lambda\in \C_+}\frac{\im\, m_+^r(\lambda)}{|m_+^r(\lambda)-m_+^r(-\lambda)|}<+\infty.
\end{equation}
\item[(vi)]
\begin{equation}\label{eq:kostenko_imre}
\sup_{y>0}\frac{\im\, m_+^r(\I y)}{|\re\, m_+^r(\I y)|}<+\infty.
\end{equation}
\item[(vii)] There is $\theta\in (0,\frac{\pi}{2})$ such that
\be\label{Everitt-m-function_3}
-\im \big( \lambda \, m_+^r(\lambda)\big)\ge 0,\qquad ( \lambda\in \Gamma_{\theta}),
\ee
where $\Gamma_{\theta}=\{\lambda\in \C_+: \frac{|\re\, \lambda|}{|\lambda|}\le \cos  \theta\}$.
\end{enumerate}
\end{theorem}
\begin{proof}
$(i)\Rightarrow (ii)$ is well known. Theorem \ref{th:connection} establishes the implication $(ii)\Rightarrow (iii)$. $(iii)\Rightarrow (iv)$ follows from Theorem \ref{Bennewitz_HELP}. $(iv)\Rightarrow (i)$ was established in \cite[Theorem 6]{Par03}.  

Further, by Theorem \ref{kostenko}, $(ii)\Rightarrow (v)$. The implication $(v)\Rightarrow (vi)$ is obvious. $(vi)\Rightarrow (vii)$ follows from Lemma \ref{cor:3.2} and \eqref{eq:m_conn}. Implication $(vii)\Rightarrow (iv)$ follows from Lemma \ref{cor:3.2}. Noting that the equivalent $(ii)\Leftrightarrow (iv)$ is already established, we complete the proof.
\end{proof}
\begin{remark}
It is interesting to note that under the assumptions of Theorem \ref{th:sim_odd} the equivalence $(i)\Leftrightarrow (iii)$ was first observed in \cite[Theorem 4.3]{BF}. However, this result can be obtained as a combination of Bennewitz's criterion (Theorem \ref{Bennewitz_HELP}) and Parfenov's criterion \cite[Theorem 6]{Par03}. Namely, Parfenov proved that the eigenfunctions of \eqref{evp} with odd $r\in L^1(-1,1)$ form a Riesz basis of $L^2_{|r|}(-1,1)$ precisely if $r$ satisfies Bennewitz's condition \eqref{Parfenov}. Applying Theorem \ref{Bennewitz_HELP} and Lemma \ref{cor:3.2}, equivalence $(i)\Leftrightarrow (iii)$ follows.   
\end{remark}


\appendix

\section{Asymptotic behavior of $m$-functions}\label{sec:app}

In this appendix, we collect some information on high energy asymptotics of $m$-functions. For a detailed exposition and further results we refer to the excellent paper by Bennewitz \cite{Ben89}.

\subsection{Weyl--Titchmarsh $m$-functions}\label{ss:a01}
Let the function $r\in L^1(0,1)$ be positive on $(0,1)$. Consider the Sturm--Liouville spectral problem
\begin{align}\label{eq:a01}
&-\big(\frac{1}{r(x)}y'\big)'=\lambda\, y,\quad x\in (0,1);\\
&\qquad (r^{-1}y')(0)=(r^{-1}y')(1-)=0.\label{eq:a01}
\end{align}
Let $c(x,\lambda)$ and $s(x,\lambda)$ be the system of fundamental solutions satisfying \eqref{eq:fs} and let $\psi(x,\lambda)$ be the Weyl solution defined by \eqref{eq:weyl_s}. The Weyl--Titchmarsh $m$-function  is then given by
\be\label{eq:a03}
m(\lambda)=-\frac{\psi(0,\lambda)}{(r^{-1}\psi')(0,\lambda)}=\frac{(r^{-1}s')(1,\lambda)}{(r^{-1}c')(1,\lambda)},\quad (\lambda\notin\R).
\ee 

Firstly, it is possible (see for details \cite[\S 2]{Ben89}) to assign $m$-functions with all its usual properties to systems of equations on $[0,1]$ defined by 
\be\label{eq:a04}
\begin{cases}
u_1(x) = & u_1(0)+\int_{[0,x)}u_2(t)dR(t),\\
u_2(x)= & u_2(0)-\lambda\int_{[0,x)}u_1(t)dt,
\end{cases}
\ee
where the function $R$ is  increasing left-continuous and of bounded variation on $[0,1]$. Namely, fix a fundamental solution $U(x,\lambda)=\begin{pmatrix} c & s\\ c^{[1]} & s^{[1]}\end{pmatrix}$ of \eqref{eq:a04} satisfying the standard initial condition at $x=0$, $U(0,\lambda)=\begin{pmatrix} 1 & 0\\ 0 &1\end{pmatrix}$. 
Then define the solution $\Psi=\begin{pmatrix} \psi \\ \psi^{[1]} \end{pmatrix}$ such that
\be\label{eq:a05}
\Psi(x,\lambda):=U(x,\lambda)\begin{pmatrix} -m(\lambda) \\ 1 \end{pmatrix},\quad \psi^{[1]}(1,\lambda)=0.
\ee
The function $m$ is called the $m$-function of \eqref{eq:a04} subject to von Neumann boundary conditions. 
Notice that in the case $dR(x)=r(x)dx$ with a positive $r\in L^1(0,1)$, the $m$-functions \eqref{eq:a03} and \eqref{eq:a05} coincide. 

Further, applying the Lagrange formula, we get
\be\label{eq:a06}
\int_0^1|\psi(x,\lambda)|^2dx=\int_0^1|s(x,\lambda)-m(\lambda)c(x,\lambda)|^2dx=\frac{\im\, m(\lambda)}{\im\, \lambda},\quad (\lambda\notin\R).
\ee
Equation \eqref{eq:a06} defines {\em the Weyl circle} $C_{\rho}(z_0)=\{z\in\C:\ |z-z_0|=\rho\}$ {\em at} $\lambda$. The center and the radius of  $C_{\rho}(z_0)$ are given by 
\[
z_0(\lambda)=\frac{(s\overline{c}^{[1]} - s^{[1]}\overline{c})(1,\lambda)}{(c\overline{c}^{[1]} - c^{[1]}\overline{c})(1,\lambda)},\qquad \rho(\lambda)=\big(2|\im\, \lambda| \int_0^1|c(x,\lambda)|^2dx\big)^{-1}.
\]  
\begin{example}\label{ex:a01}
Let $dR(x)=a\delta(x)$ where $a>0$ and $\delta$ is the Dirac $\delta$-function, i.e., $R(x)=a\chi_{(0,1]}(x)$. Then \eqref{eq:a04} becomes
\[
\begin{cases}
u_1(x) = & u_1(0)+au_2(0)\chi_{(0,1]}(x)\\
u_2(x)= & u_2(0)-\lambda (u_1(0)+a u_2(0))x
\end{cases},\quad x\in(0,1).
\]
Therefore, 
\[
U(x,\lambda)=\begin{pmatrix} 1 & a\chi_{(0,1]}(x)\\ -\lambda x & 1-\lambda ax\end{pmatrix}
\]
and hence the $m$-function is given by
\[
m(\lambda)=a-\frac{1}{\lambda},\quad (\lambda\neq 0).
\]
The Weyl circle at $\lambda\in\C_+$ has its center at $z(\lambda)=a+\frac{\I}{2\im\, \lambda}$ and radius $\rho(\lambda)=\frac{1}{2\im\, \lambda}$. Note that a real point $x=a$ belongs to this circle for every $\lambda \in \C_+$. The latter is possible only in some degenerate cases. In particular, for systems \eqref{eq:a04} the following is true: {\em if the Weyl circle  at some $\lambda\in \C_+$ contains a real point $a\in\R$, then $dR=a\delta$}.   
\end{example}
 \subsection{Asymptotic behavior at $\infty$}\label{ss:a02} 
 Since  $r$ is positive, the function $R:[0,1]\mapsto \R_+$,
 \[
 R(x):=\int_0^x r(t)dt,
 \]
 is strictly increasing and absolutely continuous on $[0,1]$. Let $R^{-1}$ denote its inverse. Let also $f$ be the inverse of $1/(R^{-1})^2$. Note that $f\to 0$ as argument goes to $+\infty$. The function $f$ describes the asymptotic behavior of $m(\cdot)$ at large $\lambda$. For instance, by \cite[Theorem 3.3]{Ben89}, there are absolute constants $M,N>0$ such that
 \[
 N|\sin\arg\, \lambda|\le \frac{|m(\lambda)|}{f(|\lambda|)}\le M/|\sin\arg\, \lambda|
 \]
 holds for a sufficiently large $|\lambda|$. 
 In particular, the latter implies that 
 \[
 tf(t)\to +\infty\quad \text{as}\quad t\to +\infty.
 \]
 To see this notice that, by  \cite[Theorem 4.2]{KK1}, $\lim_{y\uparrow +\infty} y\im \, m(\I y)=\int_{\R_+}d\tau_+=\infty$.
 
 Further, for $s\in (0,1]$ define the function 
 \[
P_s(x):= \frac{R(sx)}{R(s)},\quad x\in [0,1].
 \]
 Clearly, $P_s$ maps $[0,1]$ onto $[0,1]$. By the Helly theorem, every positive sequence, which tends to $0$, has a subsequence $s_{k}$ such that $P_k:=P_{s_{k}}$ converges pointwise and boundedly to some increasing function $P_\infty:[0,1]\to [0,1]$. Define the sequence $\rho_k$ as follows $s_k:=1/(\rho_kf(\rho_k))$, $k\in\N$. Note that $\rho_k\to +\infty$ since $s_k\to 0$. Therefore, (see the proof of \cite[Lemma 4.7]{Ben89}), $ m(\rho_k\mu)/f(\rho_k)$ is asymptotically in the Weyl circle at $\mu$ of the limit system \eqref{eq:a04}, i.e., in the Weyl circle of \eqref{eq:a04} with $P=P_\infty$. This holds uniformly for $\mu$ in any compact set in $\C_+$.


\quad

{\bf Acknowledgements.} I am deeply grateful to Mark Malamud for numerous fruitful and stimulating discussions. I am also indebted to Anton Lunev for critical remarks, which have helped to improve the exposition.      
 

\end{document}